\documentclass[a4paper,10pt,reqno]{amsart}

\usepackage{spiros}
\usepackage[pdftex]{graphicx}
\usepackage[all]{xy}

\newtheorem*{ack*}{Acknowledgement}

\title{Algebraic subdivision in simplicially controlled categories}
\author{Spiros Adams-Florou}

\begin{document}

\begin{abstract}
We generalise the notion of subdivision of a finite-dimensional locally finite simplicial complex $X$ to geometric algebra, namely to the simplicially controlled categories $\A^*(X)$, $\A_*(X)$ of Ranicki and Weiss. We prove a squeezing result: a bounded chain equivalence of sufficiently algebraically subdivided chain complexes can be squeezed to a simplicially controlled chain equivalence of the unsubdivided chain complexes. Giving $X\times\R$ a bounded triangulation measured in the open cone $O(X_+)$ we use algebraic subdivision to define a functor $\mathrm{``}-\otimes\Z\mathrm{''}:\BB(\A(X))\to \BB(\A(X\times\R))$ that corresponds to tensoring with the simplicial chain complex of $\Z$ and algebraically subdividing to be bounded over $O(X_+)$. We show that $C\simeq 0 \in \BB(\A(X))$ if and only if $\mathrm{``}C\otimes\Z\mathrm{''}$ is boundedly chain contractible over $O(X_+)$. These results have applications to Poincar\'{e} duality and homology manifold detection as a finite-dimensional locally finite simplicial complex $X$ is a homology manifold if and only if it has $X$-controlled Poincar\'{e} duality. We prove a Poincar\'e duality squeezing theorem that such a space $X$ with sufficiently controlled Poincar\'{e} duality must have $X$-controlled Poincar\'{e} duality and we prove a Poincar\'{e} duality splitting theorem with the consequence that $X$ is a homology manifold if and only if $X\times\R$ has bounded Poincar\'{e} duality over $O(X_+)$.
\end{abstract}

\maketitle 

\section*{Introduction}
It is well known that a global homotopy equivalence need not be local - not all homotopy equivalences are hereditary. A classical theorem of Vietoris' can be stated as 
\begin{thm}[Vietoris '27 \cite{Vietmapthm}]
Let $f:Y\to X$ be a surjective map between compact metric spaces. If $f$ has acyclic point inverses, then $f$ induces isomorphism on homology.
\end{thm}
Since Vietoris, many people have studied surjective maps with point inverses that are well behaved in some sense, whether they be contractible, acyclic, cell-like etc. The idea is that we weaken the condition of a map being a homeomorphism where all the point inverses are precisely points to the condition where they merely have the homotopy or homology of points.

The approach of controlled topology, developed by T. Chapman, S. Ferry and F. Quinn, is to have each space equipped with a control map to a metric space with which we are able to measure distances. Typical theorems involve a concept called squeezing, where one shows that if the size of some geometric obstruction is sufficiently small, then it can be `squeezed' arbitrarily small.

The approach of bounded topology is again to have a control map, this time necessarily to an unbounded metric space $M$, but rather than focus on the arbitrarily small to focus instead on things that are bounded over $M$. An advantage of this perspective is the functoriality obtained from not having to count epsilons. 

Since the advent of controlled and bounded topology people have been studying the relationship between the two. An idea of Pedersen and Weibel is to use the open cone construction to characterise when a map $f:X\to Y$ is an $\ep$-controlled homotopy equivalence over $M$. For $M$ a subset of $S^N$ for some $N$, the open cone $O(M_+)$ is the infinite open cone in $\R^{N+1}$ obtained by taking all the rays out from the origin through points $m\in M_+ = M\cup \{\mathrm{pt}\}$ where we have added a point to $M$ for technical reasons. 

In this document we work in the setting of finite-dimensional locally finite simplicial complexes. In \cite{toppaper} a proof of the following is presented:

\begin{thm}
Let $f:X\to Y$ be a simplicial map of finite-dimensional locally finite simplicial complexes. Then the following are equivalent:
\begin{enumerate}
 \item $f$ has contractible point inverses,
 \item $f$ is a an $\ep$-controlled homotopy equivalence measured in $Y$, for all $\ep>0$,
 \item $f\times\id_\R: X\times\R \to Y\times\R$ is a bounded homotopy equivalence measured in the open cone $O(Y_+)$.
\end{enumerate}
\end{thm}
\noindent Conditions $(1)$ and $(2)$ being equivalent is essentially well known, particularly for the case of finite simplicial complexes. (see \cite{plmf} 2.18 for example.)

The goal of this paper is to develop algebra to relate the three notions of the above theorem, with an application to Poincar\'{e} duality in mind. Using the simplicial algebraic categories $\brc{\A^*(X)}{\A_*(X)}$ of Ranicki (c.f. \cite{bluebk}) defined for a simplicial complex $X$, we prove an algebraic generalisation of this theorem:

\begin{thm}\label{above}
Let $X$ be a finite-dimensional locally finite simplicial complex. Let $C\in \BB(\A(X))$ where $\A(X)$ denotes either $\A^*(X)$ or $\A_*(X)$ and $\BB(\A)$ denotes the category of finite chain complexes in $\A$. Then the following are equivalent:
\begin{enumerate}
 \item $C(\sigma)\simeq 0 \in \A$ for all $\sigma \in X$, i.e. $C$ is locally contractible over each simplex in $X$, 
 \item $C\simeq 0 \in \A\hspace{-1mm}*\hspace{-1mm}(X)$, i.e. $C$ is globally contractible over $X$,
 \item ``$C\otimes\Z$''$\in \BB(\A(X\times \R))$ is chain contractible in $\GG_{X\times\R}(\A)$ with finite bound measured in $O(X_+)$\footnote{Here we give $X\times\R$ a bounded triangulation when measured in the open cone.}
\end{enumerate}
where $\GG_{X\times\R}(\A)$ is the $X$-graded category of \cite{bluebk}.
\end{thm}

Applying this to the algebraic mapping cone of a Poincar\'{e} duality chain equivalence we get the following corollary for Poincar\'{e} duality:

\begin{cor}\label{cor}
Let $X$ be a simplicial complex as in the statement of Theorem \ref{above}, then the following are equivalent:
\begin{enumerate}
 \item $X$ has $\ep$-controlled Poincar\'{e} duality for all $\ep>0$ measured in $X$. 
 \item $X\times\R$ has bounded Poincar\'{e} duality measured in $O(X_+)$. 
\end{enumerate}
\end{cor}
In particular condition $(1)$ is equivalent to $X$ being a homology manifold, so this gives us a way to detect homology manifolds.

\vspace{5mm}
The categories $\A(X)$ capture algebraically the concept ``$\ep$-controlled for all $\ep>0$'' in the following sense: we define an algebraic subdivision functor $Sd:\A(X) \to \A(Sd\, X)$ for these categories generalising the effect of barycentric subdivision on the simplicial chain and cochain complexes of $X$. If $X$ is finite-dimensional and locally finite, then a chain complex $C$ in $\BB(\A(X))$ has bound at most $\mathrm{mesh}(X)$ and so the bound of $Sd^i\, C \in \BB(\A(Sd^i\, X))$ is at most $\left(\dfrac{\mathrm{dim}(X)}{\mathrm{dim}(X)+1}\right)^i\mathrm{mesh}(X)$ which tends to zero as $i\to \infty$. The subdivision of a chain complex can be reassembled to give back a chain complex equivalent to one started with, thus any chain complex $C\in \BB(\A(X))$ can be given a representative with arbitrarily small bound in $\BB(\A(Sd^i\,X))$ for 
$i$ sufficiently large. 
Applying this to Poincar\'{e} duality, if we can show that a simplicial complex has Poincar\'{e} duality in $\A(X)$ then it necessarily has $\ep$-controlled Poincar\'{e} duality for all $\ep>0$, thus making it necessarily a homology manifold. However, apriori we do not necessarily know if a Poincar\'{e} duality space has Poincar\'{e} duality in $\A(X)$. 

A version of squeezing holds for these categories, namely that for a well behaved simplicial complex $X$ there exists an $\ep(X)$ such that if two chain complexes $C, D\in \BB(\A(X))$ are $\ep$-controlled chain equivalent (not necessarily in $\A(X)$) after subdividing them sufficiently, then they are chain equivalent in $\A(X)$ without subdividing. 
A consequence of this is that a simplicial complex with sufficiently small Poincar\'{e} duality will necessarily be a homology manifold. 

Again the open cone can be used to characterise when a chain complex is chain contractible in $\A(X)$. This turns out to be precisely when the chain complex ``$C\otimes\Z$'' in $\BB(\A(X\times\R))$ is $X\times\R$-graded chain contractible with finite bound measured in $O(X_+)$, where we give $X\times\R$ a simplicial decomposition with uniformly bounded simplices measured in $O(X_+)$. This is our algebraic analogue of the theorem above.

The key trick is that a chain equivalence between chain complexes in $\A(X\times\R)$ measured in $O(X_+)$ can be translated exponentially towards $\{-\infty\}$ with the effect of decreasing its bound by a scale factor. Combining this with the squeezing theorem allows us to obtain a chain equivalence over a slice $X\times\{t\}$ for $t$ large enough. The fact that the metric increases in the open cone as you go towards $\{+\infty\}$ in $\R$ and the fact that the chain equivalence had finite bound over $O(X_+)$ to begin with means that the chain equivalences on the slices $X\times \{t\}$ have control proportional to $\frac{1}{t}$ measured in $X$. 

This gives corollary \ref{cor} which is a simplicial complex version of the unproven footnote of \cite{epsurgthy}. 

The reason our key trick above works so well is that $f\times\id$ is already a product so by translating it in the $\R$ direction we do not change anything. A natural continuation to the work presented here is to tackle splitting problems where this sliding approach will not work, in which case we expect there to be a $K$-theoretic obstruction over each simplex. 

Section \ref{section:prelims} recaps some preliminaries. In section \ref{section:geomcatandassembly} the geometric categories we work and various assembly functors for them are defined. In section \ref{section:algsubdiv} we define the main construction of this paper - the algebraic subdivision functors and in section \ref{section:algsubdivprops} we prove some useful properties of these functors. In section \ref{section:squeezing} we prove the squeezing theorem. In section \ref{section:openconeandsplitting} we introduce the open cone, define a functor corresponding to tensoring with the real line and consider splitting problems. Finally, in section \ref{section:PD} we apply all the methods of the paper to studying Poincar\'{e} duality and homology manifolds.

\begin{ack*}
This work is partially supported by Prof. Michael Weiss' Humboldt Professorship. The author would like to thank the Max-Planck-Institute in Bonn for its hospitality where parts of this work have been carried out.
\end{ack*}

\section{Preliminaries}\label{section:prelims}
\subsection{Geometry}

This paper is concerned with finite-dimensional locally finite simplicial complexes (f.d. l.f. simplicial complexes). We consider abstract (not necessarily embedded in $\R^N$) simplicial complexes with a metric defined as follows. Let $\Delta^n$ denote the \textit{standard $n$-simplex} which is defined to be the convex hull of the points $(1,0,\ldots,0),(0,1,0,\ldots, 0), \ldots, (0,\ldots,0,1) \in \R^{n+1}$. Let $d_{\Delta^n}$ denote the subspace metric inherited from the standard $\ell_2$-metric on $\R^{n+1}$. An f.d. l.f. simplicial complex $X$ is given a complete metric $d_X$, which we call the \textit{standard metric}, by defining $d_X$ to be the path metric whose restriction to each $n$-simplex $\sigma\in X$ is $d_{\Delta^n}$. Distances between points in different path components are thus $\infty$. See $\S 4$ of \cite{bartelssqueezing} or Definition $3.1$ of \cite{HR95} for more details.

We shall often equip an f.d. l.f. simplicial complex $X$ with the identity control map and measure distances in $X$ with $d_X$, or given a map of such spaces $p:Y\to X$ we may measure $Y$ in $X$ with control map $p$ and metric $d_X$. In this paper we only consider the cases where $p$ is a PL map. Regarding subdivisions $X^\prime$ of $X$ we have two possible approaches: either use $\id_X: X^\prime \to (X^\prime,d_{X^\prime})$ or $\id_X: X^\prime \to (X,d_X)$. We opt for the latter as we want subdivision to make simplices smaller. 

Let $p:Y\to X$ be a PL map of f.d. l.f. simplicial complexes that is linear on each simplex $\sigma\in Y$. Then the \textit{diameter of $\sigma$ measured in $X$} is \[\mathrm{diam}(\sigma):= \sup_{x,y\in\sigma}{d_X(p(x),p(y))}.\] The \textit{radius of $\sigma$ measured in $X$} is \[\mathrm{rad}(\sigma) := \inf_{x\in\partial\sigma}d_X(p(\widehat{\sigma}),p(x)).\] The \textit{mesh of $X$ measured in $Y$} is \[\mesh(X):= \sup_{\sigma\in X}\{\mathrm{diam}(\sigma)\}.\] If mesh$(X)<\infty$ we say $X$ has a \textit{bounded triangulation}. The \textit{comesh of $X$ measured in $Y$} is \[\comesh(X):= \inf_{\sigma\in X, |\sigma|\neq 0}\{\mathrm{rad}(\sigma)\}.\] If comesh$(X)>0$ we say $X$ has a \textit{tame 
triangulation}.

We write $\sigma = v_0\ldots v_k$ for the \textit{$k$-simplex} spanned by the vertices $\{v_0,\ldots,v_k\}$:\[ v_0\ldots v_k := \{\sum_{i=0}^{k}{t_iv_i}| \sum_{i=0}^{k}t_i = 1 \}. \]
We refer to the coordinates $(t_0,\ldots,t_k)$ as \textit{barycentric coordinates} and to the point $\widehat{\sigma}:=(\dfrac{1}{k+1},\ldots,\dfrac{1}{k+1})$ as the \textit{barycentre} of $\sigma=v_0\ldots v_k$. We refer to the interior of the simplex $\sigma$, \[\mathring{\sigma}:= \{\sum_{i=0}^{k}{t_iv_i}| \sum_{i=0}^{k}t_i < 1 \}\]as the \textit{open simplex $\sigma$} and we let $|\sigma|:=k$ denotes the \textit{dimension} of $\sigma= v_0\ldots v_k$. 

Given a simplicial complex $X$, we define the \textit{Barycentric subdivision}, $Sd\, X$, of $X$ by $$Sd\, X:=\bigcup_{\sigma_0<\ldots<\sigma_k\leq X}{\widehat{\sigma}_0\ldots\widehat{\sigma}_k}.$$We denote the $i^{th}$ iterated barycentric subdivision of $X$ by $Sd^i\, X$.

For any simplex $\sigma\leq X$ we define the \textit{closed dual cell} by \[D(\sigma,X):= \{ \widehat{\sigma}_0\ldots\widehat{\sigma}_k\leq Sd\, X|\sigma\leq\sigma_0<\ldots<\sigma_k\leq X\}\] with boundary \[\partial D(\sigma,X):= \{ \widehat{\sigma}_0\ldots\widehat{\sigma}_k\leq Sd\, X|\sigma<\sigma_0<\ldots<\sigma_k\in X\}.\] We will call the interior of the closed dual cell the \textit{open dual cell} and denote it by
\[\mathring{D}(\sigma,X):= \{ \widehat{\sigma}_0\ldots\widehat{\sigma}_k\leq Sd\, X|\sigma=\sigma_0<\ldots<\sigma_k\leq X\}.\]  Note that for open dual cells $\mathring{D}(\sigma,X)= D(\sigma,X) - \partial D(\sigma,X)$.

The \textit{open star st$(\sigma)$ of a simplex $\sigma \in X$} is defined by \[\mathrm{st}(\sigma):= \bigcup_{\tau\geqslant \sigma}\mathring{\tau}.\]
The \textit{closed star St$(\sigma)$ of a simplex $\sigma\in X$} is defined by \[\mathrm{St}(\sigma):= \bigcup_{\tau\geqslant \sigma}\tau.\]
For $Y\subset X$, we denote by $\Fr Y$ the \textit{frontier of $Y$ in $X$}:\[\Fr Y:= \overline{Y}\backslash \mathring{Y}.\]

Given simplices $\sigma= v_0\ldots v_k$, $\tau= v_{k+1}\ldots v_{m}\in X$ define the \textit{join of $\sigma$ and $\tau$} by \[\sigma *\tau:= v_0\ldots v_m.\] We say that $\sigma$, $\tau\in X$ are \textit{joinable} if $\sigma*\tau\in X$. For $\sigma\in X$, define the \textit{link of $\sigma$ in $X$}, $\link(\sigma,X)$ to be the union of all simplices in $X$ that are joinable with $\sigma$.

For a simplex $\sigma$ embedded in Euclidean space and measured there via the identity map and the standard $\ell_2$ metric it is a standard result that \[ \mesh(Sd\, \sigma) \leqslant \dfrac{|\sigma|\mathrm{diam}(\sigma)}{|\sigma|+1}.\]
Dually it is also true that \[ \comesh(Sd\, \sigma) \geqslant \dfrac{\mathrm{diam}(\sigma)}{|\sigma|(|\sigma|+1)}.\] For a proof of this assertion see the appendix. Consequently, as the metric $d_X$ is standard on each simplex, measuring $Sd\,X$ in $X$ with $\id_X:Sd\,X \to (X,d_X)$ we have that 
\begin{align}
 \mesh(Sd\, X) &\leqslant \dfrac{\mathrm{dim}(X)\mesh(X)}{\mathrm{dim}(X)+1},\notag \\
 \comesh(Sd\, X) &\geqslant \dfrac{\comesh(X)}{\mathrm{dim}(X)(\mathrm{dim}(X)+1)},\notag \\
\end{align}
and so all finite iterated barycentric subdivisions $Sd^i\,X$ of $X$ have bounded and tame triangulations. Using $\id_X:X\to (X,d_X)$ as the control map $\diam (\sigma) = \sqrt{2}$ and $\rad (\sigma) = \frac{1}{\sqrt{|\sigma|(|\sigma|+1)}},$ for all $\sigma\in X$, so consequently $\mesh (X) = \sqrt{2}$ and $\comesh (X) = \frac{1}{\sqrt{\mathrm{dim}(X)(\mathrm{dim}(X)+1)}}$.

\begin{defn}
Let $\sigma \in X$ be a simplex, and $a\in\mathring{\sigma}$ a point in its interior. Define the \textit{stellar subdivision $(\sigma,a)X$ of $X$ at $a$} to be the simplicial complex obtained by replacing $\sigma * \link(\sigma,X)$ by $a * \partial \sigma *\link(\sigma,X)$. We say $(\sigma,a)X$ is a \textit{stellar subdivision of $X$}.
\end{defn}
In this paper we will only consider stellar subdivisions for which $a=\widehat{\sigma}$.

\begin{defn}
Let $\{\sigma_i\}_{i\in I}$ be a collection of simplices of $X$ with pairwise disjoint open stars: 
\[ \forall i,j\in I:\mathrm{st}(\sigma_i) \cap \mathrm{st}(\sigma_j)= \emptyset.\]
Let $X^\prime$ be the subdivision of $X$ obtained by simultaneously performing stellar subdivisions at $\widehat{\sigma}_i$ for all $i\in I$. We call $X^\prime$ a \textit{simultaneous disjoint stellar subdivision of $X$}.
\end{defn}

\begin{defn}
Let $X^\prime$ be a subdivision of $X$ obtained by performing finitely many simultaneous disjoint stellar subdivisions. Then we call $X^\prime$ an \textit{iterated stellar subdivision of $X$}.
\end{defn}

We will be primarily concerned with iterated stellar subdivisions.

\begin{ex}\label{prismexample}
The barycentric subdivision $Sd\, X$ of a finite-dimensional simplicial complex is an iterated stellar subdivision. It is the composite of $|\dim(X)|$ simultaneous disjoint stellar subdivision: first all top dimensional simplices of $X$ are stellar subdivided, then all codimension 1 simplices of $X$ and so on.
\end{ex}

\begin{defn}
Let $\mathrm{Prism}(X,X^\prime)$ denote a triangulated prism $||X||\times I$ such that the triangulation of $||X||\times \{0\}$ is $X$ and $||X||\times \{1\}$ is $X^\prime$.
\end{defn}
\begin{rmk}\label{prismremark}
Note that if $X^\prime$ is an iterated stellar subdivision of $X$, then having fixed a triangulation of $\mathrm{Prism}(X,X)$ we obtain a triangulation of $\mathrm{Prism}(X,X^\prime)$ that is an iterated stellar subdivision of $\mathrm{Prism}(X,X)$ by performing the same sequence of simultaneous disjoint stellar subdivisions on $X\times \{1\} \subset X\times [0,1]$ as we perform on $X$ to obtain $X^\prime$.
\end{rmk}

For a complete metric space $(M,d)$ the \textit{open cone} $O(M_+)$ is defined to be the identification space $M\times\R/\sim$ with $(m,t)\sim(m^\prime,t)$ for all $m,m^\prime\in M$ if $t\leqslant 0$. We define a metric $d_{O(M_+)}$ on $O(M_+)$ by setting 
\begin{eqnarray*}
 d_{O(M_+)}((m,t),(m^\prime,t)) &=& \brcc{td(m,m^\prime),}{t\geqslant 0,}{0,}{t\leqslant 0,} \\
 d_{O(M_+)}((m,t),(m,s)) &=& |t-s|
\end{eqnarray*}
 and defining $d_{O(M_+)}((m,t),(m^\prime,s))$ to be the infimum over all paths from $(m,t)$ to $(m^\prime,s)$, which are piecewise geodesics in either $M\times\{r\}$ or $\{n\}\times\R$, of the length of the path. I.e.
\[d_{O(M_+)}((m,t),(m^\prime,s)) = \max\{\min\{t,s\},0\}d_X(m,m^\prime) + |t-s|.\]
This metric is carefully chosen so that 
\[ d_{O(M_+)}|_{M\times\{t\}} = \brcc{td_{O(M_+)}|_{M\times\{1\}},}{t\geqslant 0,}{0,}{t\leqslant 0.}\]
This is precisely the metric used by Anderson and Munkholm in \cite{AndMunk} and also by Siebenmann and Sullivan in \cite{SiebSull}, but there is a notable distinction: we do not necessarily require that our metric space $(M,d)$ have a finite bound. 

Define the \textit{coning map} $j_X:X\times\R \to O(X_+)$ as the natural quotient map
\begin{eqnarray*}
X\times\R &\to& X\times\R/\sim \\ 
(x,t) &\mapsto& [(x,t)].
\end{eqnarray*}

For $M$ a proper subset of $S^n$ with the subspace metric, the open cone $O(M_+)$ can be thought of as all the points in the lines out from the origin in $\R^{n+1}$ through points in $M_+:= M \cup \{pt\}$. This is not the same as the metric we just defined above but it is Lipschitz equivalent. 

\subsection{Algebra}
\begin{defn}
Let $\A$ be an additive category.
\begin{enumerate}[(i)]
 \item An \textit{$\A$-chain complex $C$} is a sequence of objects and morphisms of $\A$
\begin{displaymath}
 \xymatrix@C=13mm{C:\,\ldots \ar[r] & C_{i+1} \ar[r]^-{(d_C)_{i+1}} & C_{i} \ar[r]^-{(d_C)_i} & C_{i-1} \ar[r] & \ldots 
}
\end{displaymath}
such that $(d_C)^2=0$. The chain complex is \textit{finite} if $\{ i\in \Z| C_i\neq 0\}$ is finite. 
 \item An \textit{$\A$-chain map $f:C\to D$} is a sequence of morphisms $f_i:C_i\to D_i$ of $\A$ such that $(d_D)_if_i = f_{i-1}(d_C)_i:C_i \to D_{i-1}$, for all $i$.
 \item Denote by $\BB(\A)$ the category of finite $\A$-chain complexes together with $\A$-chain maps.
 \item An \textit{$\A$-chain homotopy} $P:f\simeq f^\prime$ between $\A$-chain maps $f,f^\prime:C\to D$ is a sequence of morphisms $P_i:C_i \to D_{i+1}$ in $\A$ such that \[f_i-f^\prime_i = (d_D)_{i+1}P_i + P_{i-1}(d_C)_i:C_i \to D_i.\]
 \item An \textit{$\A$-chain equivalence} of $\A$-chain complexes $C,D$ is an $\A$-chain map $f:C\to D$ with an $\A$-chain homotopy inverse, i.e. an $\A$-chain map $g:D\to C$ with $\A$-chain homotopies $P_C:gf\simeq \id_C:C\to C$, $P_D:fg\simeq \id_D:D\to D$. This information will often be written more concisely as 
\begin{displaymath}
 \xymatrix@C=13mm{ (C,d_C,P_C) \ar@<0.5ex>[r]^-{f} & (D,d_D,P_D). \ar@<0.5ex>[l]^-{g}
}
\end{displaymath}

\noindent We say that the $\A$-chain complexes $C,D$ are \textit{$\A$-chain equivalent} and write $C\simeq D$ if there is an $\A$-chain equivalence $f:C\to D$. 
\item An \textit{$\A$-chain contraction} of an $\A$-chain complex $C$ is an $\A$-chain homotopy $P_C:0\simeq \id_C:C\to C$. If $C$ admits an $\A$-chain contraction, i.e. if $C\simeq 0$,  we say that $C$ is \textit{$\A$-chain contractible}.
\item The suspension $\Sigma C$ of an $\A$-chain complex $C$ is the $\A$-chain complex \[ (d_{\Sigma C})_n:= (d_C)_{n-1}: (\Sigma C)_n= C_{n-1} \to C_{n-2}=(\Sigma C)_{n-1}.\]
\item The \textit{algebraic mapping cone} of an $\A$-chain map $f:C\to D$ is the $\A$-chain complex $\C(f)$ with 
\[\C(f)_n:= C_n\oplus D_{n+1}\]
and boundary maps 
\[(d_{\C(f)})_n = \left(\begin{array}{cc} (d_C)_n & 0 \\ f_n & -(d_D)_{n+1} \end{array}\right).\]
\end{enumerate}
\end{defn}

\begin{rmk}
An $\A$-chain map $f:C\to D$ is an $\A$-chain equivalence if and only if $\C(f)$ is $\A$-chain contractible.
\end{rmk}

\section{Geometric categories and assembly}\label{section:geomcatandassembly}
In algebraic Topology the passage from topology to algebra often loses valuable geometric information. Geometric categories are designed to retain this information by having geometric information, namely a point in a topological space, associated to each piece of algebra. Roughly speaking, this enables one to keep track of where the algebra comes from. 

Let $X$ be an f.d. l.f. simplicial complex and $R$ a commutative ring. Denote by $\F(R)$ the category of finitely generated free $R$-modules.

\begin{defn}\label{geometriccategories}
\begin{enumerate}[(i)]
 \item Define the \textit{$X$-graded} category $\mathbb{G}_X(\A)$ to be the additive category whose objects are collections of objects of $\A$, $\{M(\sigma)\,|\, \sigma \in X \}$, indexed by the simplices of $X$, written as a direct sum \[\sum_{\sigma\in X} M(\sigma) \] and whose morphisms \[f = \{f_{\tau,\sigma}\}: L = \sum_{\sigma\in X}L(\sigma) \to M = \sum_{\tau\in X}M(\tau) \]are collections $\{f_{\tau,\sigma}:L(\sigma) \to M(\tau)\,|\,\sigma,\tau \in X \}$ of morphisms in $\A$ such that for each $\sigma\in X$, the set $\{\tau\in X\,|\, f_{\tau,\sigma}\neq 0 \}$ is finite. 
 
The composition of morphisms $f:L\to M$, $g:M\to N$ in $\mathbb{G}_X(\A)$ is the morphism $g\circ f:L\to N$ defined by \[(g\circ f)_{\rho, \sigma} = \sum_{\tau\in X} g_{\rho,\tau}f_{\tau,\sigma}: L(\sigma) \to N(\rho) \]where the sum is actually finite.
 \item Let $\left\{\begin{array}{c} \A^*(X) \\ \A_*(X) \end{array}\right.$ be the additive category with objects $M$ in $\mathbb{G}_X(\A)$ and with morphisms $f:M\to N$ in $\mathbb{G}_X(\A)$ such that $f_{\tau,\sigma}:M(\sigma)\to N(\tau)$ is $0$ unless $\left\{ \begin{array}{c} \tau\leqslant \sigma \\ \tau\geqslant \sigma. \end{array}\right.$ 
\end{enumerate}
\end{defn}

It is convenient to regard an $X$-graded morphism $f$ as a matrix with one column $\{f_{\tau,\sigma}\,|\,\tau\in X \}$ for each $\sigma\in X$ (containing only finitely many non-zero entries) and one row $\{f_{\tau,\sigma}\,|\,\sigma\in X \}$ for each $\tau\in X$. Morphisms of $\left\{\begin{array}{c} \A^*(X) \\ \A_*(X) \end{array}\right.$ are to be thought of as triangular matrices.

\begin{notn}
In the case where it doesn't matter which category is considered we will write $\A(X)$ to mean \textit{either $\A^*(X)$ or $\A_*(X)$}. Similarly $\A(X) \to \A(Sd\, X)$ will mean \textit{either $\A^*(X) \to \A^*(Sd\, X)$ or $\A_*(X) \to \A_*(Sd\, X)$}. 
\end{notn}

\begin{ex}
Taking locally finite chains in the case of the simplicial chain complex, the simplicial $\brc{\mathrm{chain}}{\mathrm{cochain}}$ complex $\brc{\Delta^{lf}_*(X)}{\Delta^{-*}(X)}$ is naturally a chain complex in $\brc{\BB(\A^*(\F(\Z)))}{\BB(\A_*(\F(\Z)))}$ with \[\begin{array}{rcl} \Delta^{lf}_*(X)(\sigma)= &\Delta_*(\sigma,\partial\sigma) = & \Sigma^{|\sigma|}\Z \\ \Delta^{-*}(X)(\sigma) = &\Delta^{-*}(\sigma,\partial\sigma) = &\Sigma^{-|\sigma|}\Z. \end{array}\]
\end{ex}

\begin{defn}\label{bdf2}
Let $(X,p)$ be a finite-dimensional locally finite simplicial complex with control map $p:X\to (M,d)$. 

Define the \textit{bound} of a chain map $f:C\to D$ of $X$-graded chain complexes by \[\bd(f):= \sup_{f_{\tau,\sigma}\neq 0}d(p(\widehat{\sigma}),p(\widehat{\tau})).\] 

Define the \textit{bound} of a chain homotopy $P:C_*\to D_{*+1}$ of $X$-graded chain complexes by \[\bd(P):= \sup_{P_{\tau,\sigma}\neq 0}d(p(\widehat{\sigma}),p(\widehat{\tau})).\] 

We say that a chain equivalence $f:C\to D$ of $X$-graded chain complexes has \textit{bound} $\ep$ if there exists a chain inverse $g$ and chain homotopies $P:\id_C\simeq g\circ f:C_*\to C_{*+1}$, $Q:\id_D\simeq f\circ g:D_*\to D_{*+1}$ all with bound at most $\ep$.
\qed\end{defn}

\begin{rmk}\label{makesmall}
When we measure in $X$ with the identity map as the control map, then the bound of a chain complex (or a chain equivalence) in $\A(X)$ is at most the maximum diameter of any simplex in $X$, i.e.\ $\mesh(X)$. Thus by subdividing we can get a chain complex with control as small as we like that when reassembled is chain equivalent in $\A(X)$ to the one we started with.
\end{rmk}

The following bounded categories are due to Pedersen and Weibel:
\begin{defn}\label{bddcat}
Given a metric space $(X,d)$ and an additive category $\A$, let $\CC_X(\A)$ be the category whose objects are collections $\{M(x)\,|\,x\in X \}$ of objects in $\A$ indexed by $X$ in a locally finite way, written as a direct sum \[M = \sum_{x\in X} {M(x)}\] where $\forall x\in X$, $\forall r>0$, the set $\{ y\in X \,|\, d(x,y)<r\;\mathrm{and}\; M(y)\neq 0 \}$ is finite. A morphism of $\CC_X(\A)$, \[f = \{f_{y,x}\}: L = \sum_{x\in X}L(x) \to M = \sum_{y\in X}M(y) \] is a collection $\{f_{y,x}:L(x) \to M(y)\,|\,x,y\in X\}$ of morphisms in $\A$ such that each morphism has a bound $k=k(f)$, such that if $d(x,y)>k$ then $f_{x,y} = f_{y,x} = 0$.
\qed\end{defn}

\begin{defn}
Let $C\in \BB(\A(X))$. Define the chain complex $C(\sigma)\in \BB(\A)$ by \[ C(\sigma)_n:= C_n(\sigma),\quad (d_{C(\sigma)})_n:= ((d_{C})_n)_{\sigma,\sigma}.\] Note that this is a chain complex because \[0 = ((d_C^2)_n)_{\sigma,\sigma} = \sum_{\sigma\leqslant \tau \leqslant \sigma}{((d_C)_n-1)_{\sigma,\tau}\circ ((d_C)_n)_{\tau,\sigma}} = ((d_C)_{n-1})_{\sigma,\sigma}\circ ((d_C)_n)_{\sigma,\sigma}.\]
\end{defn}

\begin{defn}\label{definesupports}
Let $C$ be in $\BB(\GG_X(\A))$, $\BB(\A^*(X))$ or $\BB(\A_*(X))$. Define the \textit{support of $C$} by
\[\mathrm{Supp}(C):= \bigcup_{\sigma\in X: C(\sigma)\neq 0}\mathring{\sigma}.\]
\end{defn}
One could think of $X$-graded chain complexes as being supported on barycentres of simplices, we take the convention that they are supported on \textit{open} simplices.

The following proposition demonstrates possibly the most important property of the categories $\A(X)$ that ``local = global'', namely a chain complex in $\BB(\A(X))$ is globally contractible if and only if it is locally contractible over each simplex. The proof of this proposition is analogous to the proof in linear algebra of the fact that a triangular matrix is invertible if and only if its diagonal entries are. In such a case one can simply write down the inverse and as we will see below given local chain contractions one can simply write down a global one.

\begin{prop}\label{chcont}
Let $X$ be a locally finite simplicial complex, and let $C$ be a chain complex in $\A(X)$. Then, 
\begin{enumerate}[(i)]
 \item $C$ is chain contractible in $\A(X)$ if and only if $C(\sigma)$ is chain contractible in $\A$ for all $\sigma\in X$.
 \item A chain map $f:C\to D$ of chain complexes in $\A(X)$ is a chain equivalence if and only if $f_{\sigma,\sigma}:C(\sigma)\to D(\sigma)$ is a chain equivalence in $\A$ for all $\sigma\in X$.
\end{enumerate}
\end{prop}
This is a well known result (\cite{bluebk} Prop. 4.7.) for which we present a new direct proof.
\begin{proof}
Part $(ii)$ follows from applying $(i)$ to the algebraic mapping cone of $f$, so it suffices to prove $(i)$.

($\Rightarrow$)
Suppose $P:C\simeq 0$ is a chain contraction in $\A(X)$. Then the diagonal entries $P_{\sigma\sigma}$ are chain contractions in $\A$ for each $C(\sigma)$.

($\Leftarrow$)
Suppose $P_{\sigma}:C(\sigma)\simeq 0$ is a chain contraction in $\A$ for each $\sigma\in X$. Then

\[P_{\tau,\sigma} := \sum_{i=0}^{|\sigma|-|\tau|}\sum_{\tau=\sigma_0<\ldots<\sigma_i=\sigma}(-1)^iP_{\sigma_0}(d_C)_{\sigma_0\sigma_1}P_{\sigma_1}\ldots (d_C)_{\sigma_{i-1}\sigma_i}P_{\sigma_i}\] defines a chain contraction in $\A^*(X)$ for $C\in\BB(\A^*(X))$ and \[P_{\tau,\sigma} := \sum_{i=0}^{|\tau|-|\sigma|}\sum_{\sigma=\sigma_0<\ldots<\sigma_i=\tau}(-1)^iP_{\sigma_0}(d_C)_{\sigma_0\sigma_1}P_{\sigma_1}\ldots (d_C)_{\sigma_{i-1}\sigma_i}P_{\sigma_i}\] defines one in $\A_*(X)$ for $C\in\BB(\A_*(X))$. 
\end{proof}

Chain complexes in $\GG_X(\A)$, $\CC_X(\A)$, $\A^*(X)$ or $\A_*(X)$ carry lots more information than chain complexes in $\A$ as we have geometric data associated to each piece of algebra. One can of course forget some of this information, and this can be done in many different ways.

\begin{defn}\label{restrictioncomplex}
Let $Y$ be a set of open simplices in $X$.
\begin{enumerate}[(i)]
 \item For $M$ an object in $\mathbb{G}_{X}(\A)$, define the \textit{restriction of $M$ to $Y$} to be the object $M|_Y$ in $\mathbb{G}_{X}(\A)$ given by \[(M|_Y)(\sigma):= \brcc{M(\sigma),}{\mathring{\sigma}\in Y}{0,}{\mathrm{otherwise}.}\]
 \item For a morphism $f:M\to N$ in $\mathbb{G}_{X}(\A)$, the \textit{restriction of $f$ to $Y$} is the morphism $f|_Y: M|_Y \to N|_Y$ in $\mathbb{G}_{X}(\A)$ defined by \[ (f|_Y)_{\tau,\sigma} := \brcc{f_{\tau,\sigma},}{\mathring{\sigma},\mathring{\tau}\in Y}{0,}{\mathrm{otherwise}.}\]
 \item For $C$ a chain complex in $\mathbb{G}_{X}(\A)$, define the \textit{restriction of $C$ to $Y$} by 
 \begin{eqnarray*}
  (C|_Y)_n &:=& C_n|_Y \\
  (d_{C|_Y})_n &:=& (d_C)_n|_Y.
 \end{eqnarray*}
\end{enumerate}
\end{defn}
In general $C|_Y$ is not a chain complex, but for $C\in\BB(\A(X))$ choosing $Y$ carefully it is:

\begin{lem}
Let $C\in\BB(\A(X))$ and let $Y$ be a set of open simplices in $X$ such that 
\begin{equation}\label{five}
\mathring{\rho},\mathring{\sigma} \in Y: \; \rho\leqslant \sigma\quad \Rightarrow \quad \mathring{\tau}\in Y \;\,  \forall \rho\leqslant \tau\leqslant \sigma.
\end{equation} 
Then $C|_Y\in\BB(\A(X))$.
\end{lem}

\begin{proof}
Consider $\A^*(X)$. For all $\rho\leqslant\sigma$, \[(d_C^2)_{\rho,\sigma} = \sum_{\rho\leqslant\tau\leqslant\sigma}{(d_C)_{\rho,\tau}(d_C)_{\tau,\sigma}}=0.\] Suppose $\mathring{\rho}$, $\mathring{\sigma}\in Y$ with $\rho\leqslant\sigma$. If $(\ref{five})$ holds then no terms are missing from the sum, so \[(d_{C|Y}^2)_{\rho,\sigma} = (d_C^2)_{\rho,\sigma} = 0.\] Similarly for $\A_*(X)$. 
\end{proof}

\begin{defn}\label{assembleobjects}
Let $Y$, $Y^\prime$ be finite sets of open simplices in $X$. 
\begin{enumerate}[(i)]
 \item For an object $M$ in $\mathbb{G}_{X}(\A)$ define the \textit{assembly of $M$ over $Y$} to be the object $M[Y]$ in $\A$ defined by \[M[Y] := \sum_{\mathring{\sigma}\in Y}{M(\sigma)}.\]
 \item For a morphism $f:M\to N$ in $\mathbb{G}_{X}(\A)$ define the \textit{assembly of $f$ from $Y$ to $Y^\prime$} to be the morphism $f_{[Y^\prime],[Y]}: M[Y] \to N[Y^\prime]$ in $\A$ given by the matrix \[ f_{[Y^\prime],[Y]} := \{f_{\sigma,\tau}\}_{\mathring{\sigma}\in Y^\prime,\mathring{\tau}\in Y}\] with respect to the direct sum decompositions of $M[Y]$ and $N[Y^\prime]$ given in $(i)$.
 \item For a chain complex $C\in\BB(\A(X))$ and a set of open simplices $Y$ satisfying $(\ref{five})$ the \textit{assembly of $C$ over $Y$}, denoted $C[Y]$, defined by 
\begin{eqnarray*}
 C[Y]_n &:=& C_n[Y] \\
 (d_{C[Y]})_n &:=& ((d_C)_n)_{[Y],[Y]}
\end{eqnarray*}
is in $\BB(\A)$. 
\end{enumerate}
\end{defn}

\begin{defn}
Let $X,Y$ be f.d. l.f. simplicial complexes. An \textit{$X$-partition} of $Y$ is a collection $\{Y_\sigma|\sigma\in X\}$ of subspaces $Y_\sigma\subseteq Y$, each a finite union of open simplices in $Y$, indexed by simplices $\sigma\in X$ such that
\begin{enumerate}
 \item $Y_\sigma \cap Y_\tau = \emptyset, \quad \forall \sigma\neq \tau \in X$,
 \item $\bigcup_{\sigma\in X}{Y_\sigma} = Y$.
\end{enumerate}
We call the $X$-partition of $Y$ $\brc{\mathrm{\textit{covariant}}}{\mathrm{\textit{contravariant}}}$ if $\overline{Y}_\sigma\cap Y_\tau \neq \emptyset$ if and only if $\brc{\tau\leqslant \sigma}{\sigma \leqslant \tau.}$
\end{defn}

\begin{lem}
Let $\{Y_\sigma|\sigma\in X\}$ be a covariant or contravariant $X$-partition of $Y$. Then $Y_\sigma$ satisfies condition $(\ref{five})$ for all $\sigma \in X$.
\end{lem}
\begin{proof}
Let $\widetilde{\rho}\leqslant \widetilde{\tau}\leqslant \widetilde{\sigma}$ in $Y$ with $\mathring{\widetilde{\rho}}$, $\mathring{\widetilde{\sigma}}\in Y_\sigma$ and $\mathring{\widetilde{\tau}}\in Y_\tau$. By the definition of covariance and contravariance, we must have $\sigma\leqslant \tau \leqslant \sigma$ as $\mathring{\widetilde{\tau}}\subset \overline{Y}_\sigma$ and $\mathring{\widetilde{\rho}}\subset \overline{Y}_\tau$. Whence $\tau=\sigma$ and condition $(\ref{five})$ holds.
\end{proof}

\begin{cor}\label{asspart}
A covariant $X$-partition of $Y$ defines an assembly functor \[\GG_Y(\A) \to \GG_X(\A)\] by assembling each $Y_\sigma$ to $\mathring{\sigma}$. This functor restricts to an assembly functor \[\A(Y)\to \A(X).\]
A contravariant $X$-partition of $Y$ defines an assembly functor \[\GG_Y(\A) \to \GG_X(\A)\] by assembling each $Y_\sigma$ to $\mathring{\sigma}$. This functor restricts to an assembly functor \[\brc{\A^*(Y)\to \A_*(X)}{\A_*(Y)\to \A^*(X).}\]
\end{cor}

\begin{rmk}\label{covcontra}
\begin{itemize}
 \item The $X$-partition of $Sd\, X$ into open dual cells, $\{\mathring{D}(\sigma, X)|\sigma \in X\}$, is contravariant.
 \item Let $Y$ be any subdivision of $X$ and $r:Y\to X$ any simplicial surjection. Then the $X$-partition of $Y$, $\{r^{-1}(\mathring{\sigma})|\sigma \in X\}$, is covariant.  
\end{itemize}
\end{rmk}

\begin{defn}\label{assemblyfunctors}
\begin{enumerate}[(i)]
 \item Let $r:Y \to X$ be any surjective simplicial map. Then $r$ induces an assembly functor $\rr_r: \mathbb{G}_{Y}(\A) \to \mathbb{G}_{X}(\A)$ defined by:
\begin{itemize}
 \item The assembly $\rr_r(M)$ of an object $M$ in $\mathbb{G}_{Y}(\A)$ is the object in $\mathbb{G}_{X}(\A)$ defined by \[\rr_r(M)(\sigma):= M[r^{-1}(\mathring{\sigma})].\]
 \item The assembly $\rr_r(f):\rr_r(M)\to \rr_r(N)$ of a morphism $f:M \to N$ in $\mathbb{G}_{Y}(\A)$ is the morphism in $\mathbb{G}_{X}(\A)$ defined by \[ ((\rr_r(f))_n)_{\tau,\sigma} := (f_n)_{[r^{-1}(\mathrm{\sigma})],[r^{-1}(\mathrm{\tau})]}.\] 
\end{itemize}
By Corollary \ref{asspart} and Remark \ref{covcontra}, $\rr_r$ restricts to give assembly functors
\[ \rr_r: \A(Y) \to \A(X).\]
 \item Define the assembly functor \[ \ttt: \mathbb{G}_{Sd\, X}(\A) \to \mathbb{G}_{X}(\A)).\] in the same manner as above by assembling each open dual cell $\mathring{D}(\sigma, X)$ in $Sd\, X$ to $\mathring{\sigma}$. By Corollary \ref{asspart} and Remark \ref{covcontra} this restricts to give functors \[ \ttt: \brc{\A^*(Sd\, X) \to \A_*(X)}{\A_*(Sd\, X) \to \A^*(X)}.\] 
\end{enumerate}
Both $\rr_r$ and $\ttt$ extend to $\BB$ of these categories in all the above cases.
 \item Given assembly functors $\rr_i:\GG_{X^\prime}(\A) \to \GG_{X}(\A)$ for $i=1,2$ defined by assembling the following $X$-partitions of $X^\prime$ \[\{Y_i(\sigma)|\sigma\in X\}\] and given a chain equivalence 
\begin{displaymath}
 \xymatrix@C=13mm{ (C,d_C,P_C) \ar@<0.5ex>[r]^-{f} & (D,d_D,P_D). \ar@<0.5ex>[l]^-{g}
}
\end{displaymath}
in $\GG_{X^\prime}(\A)$ define the \textit{assembly of this chain equivalence from $\rr_1$ to $\rr_2$} to be the chain equivalence in $\GG_{X}(\A)$ given by
\begin{displaymath}
 \xymatrix@C=13mm{ (\rr_1(C),d_{\rr_1(C)},(P_C)_{\rr_1,\rr_1}) \ar@<0.5ex>[r]^-{(f)_{\rr_2,\rr_1}} & (\rr_2(D),d_{\rr_2(D)},(P_D)_{\rr_2,\rr_2}). \ar@<0.5ex>[l]^-{(g)_{\rr_1,\rr_2}}
}
\end{displaymath}
where \[(((f)_{\rr_2,\rr_1})_n)_{\tau,\sigma}:= (f_n)_{[Y_2(\tau)],[Y_1(\sigma)]}\] and similarly for $g$, $P_C$ and $P_D$.
\end{defn}

\begin{rmk}\label{rrfunctorialwrtr}
Let $r=r_1\circ r_2$ be the composition of two surjective simplicial maps. Then \[\rr_{r_1}\circ \rr_{r_2} = \rr_{r}.\]
\end{rmk}

\section{Algebraic Subdivision}\label{section:algsubdiv}
Let $X^\prime$ be an iterated stellar subdivision of an f.d. l.f. simplicial complex. In this section we define an algebraic subdivision functor $Sd_r: \BB(\A(X)) \to \BB(\A(X^\prime))$ that generalises the effect that geometric subdivision has on the simplicial chain and cochain complexes of $X$ when considered as geometric chain complexes in $\BB(\A(X))$. In defining $Sd_r$ we use the crucial property of Proposition \ref{chcont} that two chain complexes in $\BB(\A(X))$ are chain equivalent if and only if they are locally chain equivalent in $\BB(\A)$ over each simplex of $X$. This means that one may replace each local $C(\tau)$ with a chain complex that is chain equivalent to $C(\tau)$, but distributed more finely over $X^\prime$. More precisely we view $C(\tau)\in \brc{\BB(\A^*(X))}{\BB(\A_*(X))}$ as \[ C(\tau)\otimes_\Z \Z = \brc{C(\tau)\otimes_\Z \Sigma^{-|\sigma|}\Delta_*(\tau,\partial\tau)}{C(\tau)\otimes_\Z \Sigma^{|\sigma|}\Delta^{-*}(\tau,\partial\tau).}\]A simplicial approximation to the identity $r:X^\prime \to X$ and a homotopy $P:\id_X \simeq r$ provide chain equivalences \[\brc{ \Delta^{lf}_{*}(X) \stackrel{\sim}{\longrightarrow} \Delta^{lf}_{*}(X^\prime) }{\Delta^{-*}(X) \stackrel{\sim}{\longrightarrow} \Delta^{-*}(X^\prime)}\] which restrict to local chain equivalences for all $\tau\in X$: \[\brc{\Delta_*(\tau,\partial\tau) \stackrel{\sim}{\longrightarrow} \Delta_{*}(X^\prime)[r^{-1}(\mathring{\tau})]}{\Delta^{-*}(\tau,\partial\tau) \stackrel{\sim}{\longrightarrow} \Delta^{-*}(X^\prime)[r^{-1}(\mathring{\tau})].}\] Thus defining $Sd_r:\BB(\A(X))\to \BB(\A(X^\prime))$ so that \[Sd_r\,C[r^{-1}(\mathring{\tau})] = \brc{C(\tau)\otimes_\Z \Sigma^{-|\sigma|}\Delta_*(\overline{r^{-1}(\mathring{\tau})}, \mathrm{Fr}\, r^{-1}(\mathring{\tau}))}{C(\tau)\otimes_\Z \Sigma^{|\sigma|}\Delta^{-*}(\overline{r^{-1}(\mathring{\tau})}, \mathrm{Fr}\, r^{-1}(\mathring{\tau}))} \]we have that for all $\tau\in X$, $\rr_r Sd_r\, C(\tau) \simeq C(\tau)$ in $\BB(\A)$ and hence $\rr_r Sd_r\, C \simeq C$ in $\BB(\A(X))$. 

We now spell out exactly how $Sd_r$ is defined to have this local form and to be functorial. As we are working with iterated stellar subdivisions we define $Sd_r$ first for a single stellar subdivision and then use this to define $Sd_r$ for the more general case. 

\begin{defn}
Let $r: (\sigma,\widehat{\sigma})X \to X$ be the simplicial approximation to the identity uniquely determined by a choice of vertex $r(\widehat{\sigma})\in \sigma$ to map $\widehat{\sigma}$ to. Define the \textit{algebraic subdivision functors} $Sd_r\,: \BB(\A(X)) \to \BB(\A(\sigma,\widehat{\sigma})X)$ by
\begin{enumerate}[(i)]
 \item for $C\in \BB(\A^*(X))$ and $\widetilde{\rho}$, $\widetilde{\sigma}\in (\sigma,\widehat{\sigma})X$ set 
\begin{align}
Sd_r\,(C)(\widetilde{\sigma})_n &= C(r(\widetilde{\sigma}))_{n-|\widetilde{\sigma}|+|r(\widetilde{\sigma})|}, \label{SdCn}
\end{align}
\begin{equation}\label{dSdconcise}
(d_{Sd_r\,C})_{\widetilde{\rho},\widetilde{\sigma},n} = \left\{ \begin{array}{cc} (d_C)_{r(\widetilde{\rho}),r(\widetilde{\sigma}),n - |\widetilde{\sigma}|+|r(\widetilde{\sigma})|}, & \widetilde{\rho}\leqslant\widetilde{\sigma},\;|\widetilde{\sigma}|-|r(\widetilde{\sigma})| = |\widetilde{\rho}|-|r(\widetilde{\rho})|, \\ (-1)^{n}\id_{C(r(\widetilde{\sigma}))_{n - |\widetilde{\sigma}|+|r(\widetilde{\sigma})|}}, & \widetilde{\rho}*r(\widehat{\sigma}) = \widetilde{\sigma}, \; \widehat{\sigma}\in \widetilde{\rho}, \\ (-1)^{n+1}\id_{C(r(\widetilde{\sigma}))_{n - |\widetilde{\sigma}|+|r(\widetilde{\sigma})|}}, & \widetilde{\rho}*\widehat{\sigma} = \widetilde{\sigma},\; r(\widehat{\sigma})\in \widetilde{\rho}, \\ 0, & \mathrm{otherwise}. \end{array} \right.
\end{equation}
\begin{equation}\label{Sdfconcise}
(Sd_rf)_{\widetilde{\rho},\widetilde{\sigma},n}:= \brcc{f_{r(\widetilde{\rho}),r(\widetilde{\sigma}),n-|\widetilde{\sigma}|+|r(\widetilde{\sigma})|,}}{\widetilde{\rho}\leqslant\widetilde{\sigma},\;|\widetilde{\sigma}|-|r(\widetilde{\sigma})| = |\widetilde{\rho}|-|r(\widetilde{\rho})|,}{0,}{\mathrm{otherwise},}
\end{equation}
\item for $C\in \BB(\A_*(X))$ and $\widetilde{\rho}$, $\widetilde{\sigma}\in (\sigma,\widehat{\sigma})X$ set 
\begin{align}
Sd_r\,(C)(\widetilde{\sigma})_n &= C(r(\widetilde{\sigma}))_{n+|\widetilde{\sigma}|-|r(\widetilde{\sigma})|}, \label{SdCnother}
\end{align}
\begin{equation}\label{dSdconciseother}
(d_{Sd_r\,C})_{\widetilde{\sigma},\widetilde{\rho},n} = \left\{ \begin{array}{cc} (d_C)_{r(\widetilde{\sigma}),r(\widetilde{\rho}),n + |\widetilde{\sigma}|-|r(\widetilde{\sigma})|}, & \widetilde{\rho}\leqslant\widetilde{\sigma},\;|\widetilde{\sigma}|-|r(\widetilde{\sigma})| = |\widetilde{\rho}|-|r(\widetilde{\rho})|, \\ (-1)^{n}\id_{C(r(\widetilde{\sigma}))_{n + |\widetilde{\sigma}|-|r(\widetilde{\sigma})|}}, & \widetilde{\rho}*r(\widehat{\sigma}) = \widetilde{\sigma}, \; \widehat{\sigma}\in \widetilde{\rho}, \\ (-1)^{n+1}\id_{C(r(\widetilde{\sigma}))_{n + |\widetilde{\sigma}|-|r(\widetilde{\sigma})|}}, & \widetilde{\rho}*\widehat{\sigma} = \widetilde{\sigma},\; r(\widehat{\sigma})\in \widetilde{\rho}, \\ 0, & \mathrm{otherwise}. \end{array} \right.
\end{equation}
\begin{equation}\label{Sdfconciseother}
(Sd_rf)_{\widetilde{\sigma},\widetilde{\rho},n}:= \brcc{f_{r(\widetilde{\rho}),r(\widetilde{\sigma}),n+|\widetilde{\sigma}|-|r(\widetilde{\sigma})|,}}{\widetilde{\rho}\leqslant\widetilde{\sigma},\;|\widetilde{\sigma}|-|r(\widetilde{\sigma})| = |\widetilde{\rho}|-|r(\widetilde{\rho})|,}{0,}{\mathrm{otherwise}.}
\end{equation}
\end{enumerate}
\end{defn}
These formulae do indeed define functors $Sd_r:\BB(\A(X))\to \BB(\A(X^\prime))$: the fact that $Sd_r\, C$ is a chain complex is more or less immediate and functoriality follows from the fact that for all $\widetilde{\rho}\leqslant\widetilde{\sigma}$ there is a one-one correspondence
\begin{displaymath}
 \xymatrix{ \left\{ \widetilde{\rho}\leqslant\widetilde{\tau}\leqslant\widetilde{\sigma} \right\} \ar@{<->}[r]^-{1:1} & \left\{ r(\widetilde{\rho})\leqslant \tau \leqslant r(\widetilde{\sigma}) \right\}.
}
\end{displaymath}

\begin{rmk}
Note that the definition of $Sd_r$ depends on a choice of simplicial approximation to the identity $r:X^\prime \to X$. It is not possible to define a subdivision functor $Sd: \BB(\A(X)) \to \BB(\A(X^\prime))$ canonically but we will observe that the dependence on the choice does not matter in the cases we care about the most and in fact being able to choose $r$ will be an advantage.
\end{rmk}

\begin{rmk}
Note that \[(Sd_r\, C)|_{r^{-1}(\mathring{\tau})} = \brcc{ C(\tau) \otimes \Sigma^{-|\tau|}\Delta_*(\overline{r^{-1}(\mathring{\tau})},\mathrm{Fr}\, r^{-1}(\mathring{\tau})),}{C\in \BB(\A^*(X))}{C(\tau) \otimes \Sigma^{|\tau|}\Delta^{-*}(\overline{r^{-1}(\mathring{\tau})},\mathrm{Fr}\, r^{-1}(\mathring{\tau})),}{C\in \BB(\A_*(X)).}\] Hence it is possible to express $Sd_r\, C$ as a naive componentwise tensor product. Defining this tensor product more rigorously is another way to see the functoriality of $Sd_r$.
\end{rmk}

\begin{thm}\label{Thm:subdivassemblechainequiv}
Let $X^\prime= (\sigma,\widehat{\sigma})X$. A choice of simplicial approximation to the identity, $r:X^\prime \to X$, defines an algebraic subdivision functor $Sd_r\,: \BB(\A(X)) \to \BB(\A(X^\prime))$ such that there is a canonical chain equivalence $\rr_r Sd_r\, C\simeq C \in \BB(\A(X))$ for all $C\in \BB(\A(X))$. 
\end{thm}

\begin{proof}
We prove the Theorem for $\A^*(X)$, the proof for $\A_*(X)$ is similar. We proceed by constructing chain maps $s_*: C \to \rr_r Sd_r C$ and $r_*: \rr_r Sd_r C \to C$. Then, by Proposition \ref{chcont}, it suffices to find local chain homotopies in $\A$: $\id_{C(\tau)} \simeq (r_*)_{\tau,\tau}(s_*)_{\tau,\tau}$ and $\id_{\rr_r Sd_r C(\tau)} \simeq (s_*)_{\tau,\tau}(r_*)_{\tau,\tau}$ for all $\tau \in X$.

Define the chain map $s_*: C \to \rr_r Sd_r C$ by setting $(s_*)_{\rho,\tau,n}: C(\tau)_n \to \rr_r Sd_r C (\rho)_n$ to be the map with component to $Sd_r C (\widetilde{\rho})$ given by
\[(s_*)_{\widetilde{\rho},\tau,n} = \left\{ \begin{array}{cc} (\id_{C(\tau)})_n, & \widetilde{\rho} = \tau,\;\mathrm{or}\; \sigma\leqslant r(\widetilde{\rho}) = \tau \, \mathrm{and} \, |\widetilde{\rho}|=|\tau|, \\ (-1)^{n+1}(d_C)_{r(\widetilde{\rho}),\tau,n}, & \tau = \sigma * \tildetilde{\tau},\; \widetilde{\rho} = \widehat{\sigma}*r(\widehat{\sigma})*\tildetilde{\sigma}*\tildetilde{\rho},\; \tildetilde{\sigma}<\sigma,\; \tildetilde{\rho}\leqslant \tildetilde{\tau} \in \link(\sigma,X), \\ 0, & \mathrm{otherwise.}\end{array} \right.\]
Note that the component $(s_*)_{\tau,\tau,n}$ has the following form for all $\tau\in X$:
\begin{displaymath}
 \hspace{-15mm}\xymatrix@R=1mm@C=0mm{ (s_*)_{\tau,\tau,n}:C(\tau)_n \ar@{}[d]|-{\quad\equalsdown \quad\quad\quad \equalsdown }  \ar[rrrrr] &&&&& \rr_r Sd_r\, C(\tau)_n \ar@{}[d]|-{\equalsdown} \\
 \quad\quad\quad(\id_{C(\tau)})_n \otimes (\Sigma^{-|\tau|}s_{\tau,\tau})_0:C(\tau)_n\otimes (\Sigma^{-|\tau|}\Delta^{lf}_*(\tau,\partial \tau))_0 \ar[rrrrr] &&&&& C(\tau)_n \otimes(\Sigma^{-|\tau|}\Delta^{lf}_*(\overline{I_\tau},\mathrm{Fr} I_\tau))_0
}
\end{displaymath}
where $s=(\id_X)_*:\Delta^{lf}_*(X)\to \Delta^{lf}_*(X^\prime)$ is the induced map on locally finite simplicial chains and $I_\tau$ is shorthand for $r^{-1}(\mathring{\tau})$.

Define the chain map $r_*: \rr_r Sd_r\, C \to C$ by setting 
\begin{displaymath}
 \hspace{-10mm}\xymatrix@R=1mm@C=0mm{ (r_*)_{\tau,\tau,n}:\rr_r Sd_r\, C(\tau)_n \ar@{}[d]|-{\equalsdown \quad\quad\quad \equalsdown \quad\quad }  \ar[rrrrr] &&&&& C(\tau)_n \ar@{}[d]|-{\equalsdown} \\
 \quad(\id_{C(\tau)})_n \otimes (\Sigma^{-|\tau|}r_{\tau,\tau})_0:C(\tau)_n \otimes(\Sigma^{-|\tau|}\Delta^{lf}_*(\overline{I_\tau},\mathrm{Fr} I_\tau))_0 \ar[rrrrr] &&&&& C(\tau)_n\otimes (\Sigma^{-|\tau|}\Delta^{lf}_*(\tau,\partial \tau))_0
}
\end{displaymath}
and $(r_*)_{\rho,\tau,n} = 0$ for $\rho \neq \tau$. Written out in components this is 
\[ (r_*)_{\tau,\widetilde{\rho},n} = \brcc{\id_{C(\tau)}: Sd_r\, C(\widetilde{\rho})_n \to C(\tau)_n,}{r(\widetilde{\rho}) = \tau, \; |\widetilde{\rho}|=|\tau|,}{0,}{\mathrm{otherwise.}} \]

The verification that $s_*$ and $r_*$ are indeed chain maps is technical and rather unilluminating so this is relegated to the Appendix.

Observe that 
\begin{align}
(r_*s_*)_{\tau,\tau} &= (r_*)_{\tau,\tau}(s_*)_{\tau,\tau} \notag \\
&=id_{C(\tau)}\otimes \Sigma^{-|\tau|}r_{\tau,\tau} \circ id_{C(\tau)}\otimes \Sigma^{-|\tau|}s_{\tau,\tau} \notag \\
&= \id_{C(\tau)} \notag
\end{align}
as $r_{\tau,\tau} \circ s_{\tau,\tau} = \id_{\Delta^{lf}_*(\tau,\partial\tau)}$.

Since $r$ is a simplicial approximation to the identity, there is a $P:r\simeq \id$ which induces a chain homotopy $P:r_*s_* \simeq \id$ on simplicial chains. This restricts to a local chain homotopy $P_{\tau,\tau}: r_{\tau,\tau}s_{\tau,\tau}\simeq \id$ which in turn induces a local chain homotopy \[(\id_{C(\tau)}\otimes \Sigma^{-|\tau|}P_{\tau,\tau}:(s_*r_*)_{\tau,\tau} \simeq \id_{C(\tau)}.\]
An application of Proposition \ref{chcont} provides a $\BB(\A(X))$ chain equivalence $\rr_r Sd_r\, C \simeq C$ as required. 
\end{proof}

\begin{defn}\label{disjointstellars}
Let $X^\prime$ be a simultaneous disjoint stellar subdivision of $X$, obtained by simultaneously stellar subdividing the collection of simplices $\{\sigma_i\}_{i\in I}$. Let $r:X^\prime \to X$ be the simplicial approximation to the identity determined by a choice $r(\widehat{\sigma}_i)\in \sigma_i$ for all $i\in I$.

Define the \textit{algebraic subdivision functor} $Sd_r: \BB(\A(X)) \to \BB(\A(X^\prime))$ by 
\begin{align}
Sd_r|_{X \backslash \bigcup_{i\in I}\mathrm{st(\sigma_i)}} &:= \id, \notag \\
Sd_r|_{\mathrm{St}(\sigma_i)} &:= Sd_{r_i}|_{\mathrm{St}(\sigma_i)},\quad \forall i\in I, \notag
\end{align}
where $Sd_{r_i}$ is the algebraic subdivision functor for the single stellar subdivision $(\sigma_i,\widehat{\sigma}_i)X$, defined with $r_i:(\sigma_i,\widehat{\sigma}_i)X\to X$ given by the choice $r_i(\widehat{\sigma}_i) = r(\widehat{\sigma}_i)$.

For $X$ an iterated stellar subdivision define $Sd_r$ as the composition of the subdivision functors for each single simultaneous disjoint stellar subdivision:
\[Sd_r:= Sd_{r_n}\circ \ldots \circ Sd_{r_0},\]
where the simplicial approximation to the identity $r:X^\prime \to X$ is the composition $r= r_0 \circ r_1 \circ \ldots \circ r_n$.
\end{defn}

\begin{prop}\label{chainequivforcompositions}
Let $X^\prime$ be an iterated stellar subdivision of $X$. Then $\rr_r Sd_r\, C\simeq C \in \BB(\A(X))$
\end{prop}
\begin{proof}
It suffices to prove the result for a single simultaneous stellar subdivision as we can compose chain equivalences. For a single simultaneous stellar subdivision by Theorem \ref{Thm:subdivassemblechainequiv} we know the result holds over the open star of each simplex being subdivided and we can glue by the identity elsewhere. 
\end{proof}

The following is a straight-forward observation.
\begin{rmk}
Let $Sd_r:\BB(\A(X))\to \BB(\A(X^\prime))$ be as in Definition \ref{disjointstellars}. Then $D\simeq Sd_r\, C\in \BB(\A(X^\prime))$ if and only if $Sd_r\,\rr_r D \simeq D\in \BB(\A(X^\prime))$
\end{rmk}

\section{Properties of algebraic subdivision}\label{section:algsubdivprops}
In this section we note a few properties of the subdivision functors and verify the claim that algebraic subdivision generalises the algebraic effect that barycentric subdivision has on the simplicial chain and cochain complexes. Fix any iterated stellar subdivision $X^\prime$ of $X$ and algebraic subdivision functors $Sd_r:\BB(\A(X)) \to \BB(\A(X^\prime))$ defined using the simplicial approximation to the identity $r:X^\prime \to X$.

\begin{lem}\label{Lem:SubcontIsCont}
$C\simeq 0 \in \A(X)$ if and only if $Sd_r\, C \simeq 0 \in \A(X^\prime)$. 
\end{lem}
\begin{proof}
By Proposition \ref{chcont} and Definition \ref{disjointstellars}, 
\begin{align}
C\simeq 0 \in \A(X) &\Longleftrightarrow \forall \sigma \in X, C(\sigma)\simeq 0 \in \A \notag \\
&\Longleftrightarrow \forall \widetilde{\sigma}\in X^\prime, Sd_r\, C(\widetilde{\sigma}) = \Sigma^{|r(\widetilde{\sigma})|-|\widetilde{\sigma}|}C(r(\widetilde{\sigma})) \simeq 0 \in \A \notag \\
&\Longleftrightarrow Sd_r\, C \simeq 0 \in \A(X^\prime). \notag
\end{align}
\end{proof}

\begin{lem}\label{Sdcommuteswithcones}
The subdivision functor $Sd_r\,$ commutes with taking algebraic mapping cones: \[ Sd_r \C (f:C\to D) = \C(Sd_r\, f: Sd_r\, C \to Sd_r\, D).\]
\end{lem}
\begin{proof}
It suffices to prove the result for a single stellar subdivision as a gluing and composition argument give it in the more general case, so let $Sd_r: \BB(\A(X)) \to \BB(\A((\sigma, \widehat{\sigma})X))$. We prove the $\A^*$ case, the $\A_*$ case is analogous.

First we verify that the $n$-chains are exactly the same:
\begin{align}
Sd_r \C (f: C\to D)(\widetilde{\rho})_n &= \C(f:C\to D)(r(\widetilde{\rho}))_{n+|r(\widetilde{\rho})|-|\widetilde{\rho}|} \notag \\
&= C(r(\widetilde{\rho}))_{n+|r(\widetilde{\rho})|-|\widetilde{\rho}|} \oplus D(r(\widetilde{\rho}))_{n+|r(\widetilde{\rho})|-|\widetilde{\rho}|+1} \notag \\
&= (Sd_r\, C)_n(\widetilde{\rho}) \oplus (Sd_r\, D)_{n+1}(\widetilde{\rho}) \notag \\
&= \C( Sd_r\, f: Sd_r\, C\to Sd_r\, D)_n(\widetilde{\rho}). \notag 
\end{align}
\noindent Next we check the differentials are the same in the various different cases. If $\widetilde{\rho}\leqslant \widetilde{\sigma}$ with $|\widetilde{\sigma}|-|r(\widetilde{\sigma})| = |\widetilde{\rho}|-|r(\widetilde{\rho})|$ then
\begin{align}
(d_{Sd_r \C (f)})_{\widetilde{\rho},\widetilde{\sigma},n} &= (d_{\C(f)})_{r(\widetilde{\rho}),r(\widetilde{\sigma}), n+|r(\widetilde{\sigma})|-|\widetilde{\sigma}|} \notag \\
&= \matrixcc{(d_C)_{r(\widetilde{\rho}),r(\widetilde{\sigma}), n+|r(\widetilde{\sigma})|-|\widetilde{\sigma}|}}{0}{f_{r(\widetilde{\rho}),r(\widetilde{\sigma}), n+|r(\widetilde{\sigma})|-|\widetilde{\sigma}|}}{-(d_D)_{r(\widetilde{\rho}),r(\widetilde{\sigma}), n+|r(\widetilde{\sigma})|-|\widetilde{\sigma}|+1}} \notag \\
&= \matrixcc{(d_{Sd_r\, C})_{\widetilde{\rho}, \widetilde{\sigma}, n}}{0}{(Sd_r\,f)_{\widetilde{\rho}, \widetilde{\sigma}, n}}{-(d_{Sd_r\,D})_{\widetilde{\rho}, \widetilde{\sigma}, n+1}} \notag \\
&= (d_{\C(Sd_r\,f)})_{\widetilde{\rho}, \widetilde{\sigma}, n}. \notag
\end{align}
If $\widetilde{\rho}*r(\widehat{\sigma}) = \widetilde{\sigma}$, $\widehat{\sigma}\in\widetilde{\rho}$ then
\begin{align}
(d_{Sd_r \C (f: C\to D)})_{\widetilde{\rho},\widetilde{\sigma},n} &= 
(-1)^n\id_{\C(f)(r(\widetilde{\sigma}))_{n+|r(\widetilde{\sigma})|-|\widetilde{\sigma}|}} \notag \\
&= \matrixcc{(-1)^n\id_{C(r(\widetilde{\sigma}))_{n+|r(\widetilde{\sigma})|-|\widetilde{\sigma}|}}}{0}{0}{(-1)^n\id_{D(r(\widetilde{\sigma}))_{n+|r(\widetilde{\sigma})|-|\widetilde{\sigma}|}}} \notag \\
&= (d_{\C(Sd_r\,f)})_{\widetilde{\rho}, \widetilde{\sigma}, n}. \notag
\end{align}
The case where $\widetilde{\rho}*\widehat{\sigma} = \widetilde{\sigma}$ and $r(\widehat{\sigma})\in\widetilde{\rho}$ proceeds similarly.
\end{proof}

Algebraic subdivision generalises the effect that geometric subdivision has on the simplicial chain and cochain.


\begin{prop}\label{alggeneralisestop}
Let $X^\prime$ be obtained from $X$ by a single simultaneous stellar subdivision and let $Sd_r:\BB(\A(X)) \to \BB(\A(X^\prime))$ be any choice of algebraic subdivision functor. Let $C=\Delta_*^{lf}(X)$ be the simplicial chain complex of $X$ with respect to a choice of orientations $[\sigma]_X$ for all $\sigma \in X$. Then \[Sd_r C = \Delta_*^{lf}(X^\prime),\] where $\Delta_*^{lf}(X^\prime)$ is the simplicial chain complex of $X^\prime$ with the following choice of simplex orientations
\begin{align}
[\tau]_{X^\prime} &= [\tau]_X,\quad\quad\quad\;\;\,  \widehat{\sigma}\notin \tau, \notag \\
[\widetilde{\rho}*\widehat{\sigma}]_{X^\prime} &= -[[\widetilde{\rho}]_X,\widehat{\sigma}], \quad\quad\, r(\widehat{\sigma})\in \widetilde{\rho},\; \widehat{\sigma}\notin\widetilde{\rho}, \notag \\
[\widetilde{\rho}*\widehat{\sigma}]_{X^\prime} &= [\widetilde{\rho}*r(\widehat{\sigma})]_X, \quad r(\widehat{\sigma}),\,\widehat{\sigma} \notin \widetilde{\rho}. \notag
\end{align}
The same is true of the simplicial cochain complex.
\end{prop}

\begin{proof}
First we verify that the $n$-chains are the same.
\begin{align}
 Sd_r\, C(\widetilde{\rho})_* &= C(r(\widetilde{\rho}))_{* + |r(\widetilde{\rho})|-|\widetilde{\rho}|} \notag \\
 &= (\Sigma^{|\widetilde{\rho}|-|r(\widetilde{\rho})|}C(r(\widetilde{\rho})))_* \notag \\
 &= (\Sigma^{|\widetilde{\rho}|-|r(\widetilde{\rho})|}\Sigma^{|r(\widetilde{\rho})|}\Z)_* \notag \\
 &= (\Sigma^{|\widetilde{\rho}|}\Z)_* \notag \\
 &= \Delta^{lf}_*(X^\prime)(\widetilde{\rho}). \notag
\end{align}
Next we examine the boundary maps in various cases. Suppose first that $\widetilde{\rho}\leqslant \widetilde{\tau}$ with $|r(\widetilde{\rho})|-|\widetilde{\rho}| = |r(\widetilde{\tau})|-|\widetilde{\tau}|$. Then
\[ (d_{Sd_r\,C})_{\widetilde{\rho},\widetilde{\tau},n} = (d_C)_{r(\widetilde{\rho}),r(\widetilde{\tau}),n+|r(\widetilde{\tau})|-|\widetilde{\tau}|} = (d_C)_{\widetilde{\rho},\widetilde{\tau},n} = (d_{\Delta^{lf}_*(X^\prime)})_{\widetilde{\rho},\widetilde{\tau},n}\] since $X^\prime|_{\widetilde{\tau}} = X|_{\widetilde{\tau}}$.

Suppose $\widetilde{\tau} = \widetilde{\rho}*\widehat{\sigma}$ with $r(\widehat{\sigma})\in \widetilde{\rho}$. Then by the choice of orientations for simplices in $X^\prime$ \[(d_{\Delta^{lf}_*(X^\prime)})_{\widetilde{\rho},\widetilde{\tau},|\widetilde{\tau}|} = -(-1)^{|\widetilde{\rho}|-1} = (-1)^{|\widetilde{\tau}|+1} = (d_{Sd_r\,C})_{\widetilde{\rho},\widetilde{\tau},|\widetilde{\tau}|}.\]

Suppose $\widetilde{\tau} = \widetilde{\rho}*r(\widehat{\sigma})$ with $\widetilde{\rho}=\widehat{\sigma}*\tildetilde{\rho}$. Then
\begin{align}
0 &=(d_{Sd_r\,C})^2_{\tildetilde{\rho},\tildetilde{\rho}*\widehat{\sigma}*r(\widehat{\sigma})} \notag \\
&= (d_{Sd_r\,C})_{\tildetilde{\rho},\tildetilde{\rho}*\widehat{\sigma}}(d_{Sd_r\,C})_{\tildetilde{\rho}*\widehat{\sigma},\tildetilde{\rho}*\widehat{\sigma}*r(\widehat{\sigma})} \notag \\
&+ (d_{Sd_r\,C})_{\tildetilde{\rho},\tildetilde{\rho}*r(\widehat{\sigma})}(d_{Sd_r\,C})_{\tildetilde{\rho}*r(\widehat{\sigma}),\tildetilde{\rho}*\widehat{\sigma}*r(\widehat{\sigma})} \notag \\
&= (d_{\Delta^{lf}_*(X^\prime)})_{\tildetilde{\rho},\tildetilde{\rho}*\widehat{\sigma}}(d_{Sd_r\,C})_{\tildetilde{\rho}*\widehat{\sigma},\tildetilde{\rho}*\widehat{\sigma}*r(\widehat{\sigma})} \notag \\
&+ (d_{\Delta^{lf}_*(X^\prime)})_{\tildetilde{\rho},\tildetilde{\rho}*r(\widehat{\sigma})}(d_{\Delta^{lf}_*(X^\prime)})_{\tildetilde{\rho}*r(\widehat{\sigma}),\tildetilde{\rho}*\widehat{\sigma}*r(\widehat{\sigma})} \notag  
\end{align}
where the last equality is given by the previous verifications. We get a similar four term expansion for $0=(d_{\Delta^{lf}_*(X^\prime)})^2_{\tildetilde{\rho},\tildetilde{\rho}*\widehat{\sigma}*r(\widehat{\sigma})}$ which when compared to the above implies that 
\[(d_{\Delta^{lf}_*(X^\prime)})_{\widetilde{\rho},\widetilde{\tau},|\widetilde{\tau}|} = (d_{Sd_r\,C})_{\widetilde{\rho},\widetilde{\tau},|\widetilde{\tau}|}.\]
The case of the simplicial cochain complex is analogous.
\end{proof}

\section{Squeezing}\label{section:squeezing}
Since $\comesh(X)>0$ if we take $i$ large enough we may construct the retracting map to retract an $\ep$-neighbourhood of each simplex back onto that simplex for $\ep<\comesh(X)$. This procedure is carefully detailed in Construction \ref{constructr} and the result is proved in Proposition \ref{technicalprop}.

\begin{constr}\label{constructr}
Let $X$ be a fixed f.d. l.f. simplicial complex. We construct an $\N$-parameter family of simplicial approximations to the identity \[ \{r_j: Sd^{j+1}\, X \to Sd^j\,X \}_{j\in \N}\] together with corresponding canonical homotopies \[\{P_j:\id_X\simeq r_j \}_{j\in\N}\] as follows.

Let $r_0:Sd\, X\to X$ be defined by any choices of $r_0(\widehat{\tau})\in \tau$ for all $\tau\in X$.

For $j\geqslant 1$ consider defining $r_j:Sd^{j+1}\, X \to Sd^j\, X$. For all $\tau \in Sd^j\, X$ we must select a vertex $r_j(\widehat{\tau})\in \tau$ to map $\widehat{\tau}$ to. For all $\tau \in Sd^j\, X$ there is a unique simplex $\rho\in Sd\,X$ with $\mathring{\tau}\subset\mathring{\rho}$. Each $\rho\in Sd\, X$ can be written uniquely as $\widehat{\sigma}_0\ldots \widehat{\sigma}_n$ for $\sigma_0<\ldots < \sigma_n \subset X$ and $n=|\rho|$. Note that necessarily $\mathring{\tau}\subset \mathring{\sigma}_n$. If $n=0$, then $\widehat{\tau}=\tau = \rho$ is a vertex so we have no choice but to define $r_j(\widehat{\tau}) = \widehat{\tau}$. Otherwise define $r_j(\widehat{\tau})$ to be any vertex $v$ of $\tau$ which minimises the distance $d_X(v,\widehat{\sigma}_0,\ldots, \widehat{\sigma}_{n-1})$.

Since $X$ is given the standard metric, if $\sigma^\prime<\sigma$ is a codimension $1$ face, then \[d_X(x,\sigma^\prime) = d_X(x,\partial\sigma),\quad \forall x\in \sigma^\prime*\widehat{\sigma}.\] Thus $r_j$ is chosen to minimise the distance to $\partial\sigma$ for $j\geqslant 1$. Every simplex $\widetilde{\tau}\in Sd^{j+1}\,\sigma$ is contained in $\widehat{\sigma}* Sd^j\, \sigma^\prime$ for some codimension $1$ face $\sigma^\prime<\sigma$. All vertices of $\widetilde{\tau}$ are mapped by $r_j$ closer to $\sigma^\prime$ (and hence the boundary). Thus, by convexity of $\widetilde{\tau}$, all points of $\widetilde{\tau}$ are mapped closer to the boundary by $r_j$ for $j\geqslant 1$. Whence 
\begin{equation}\label{rjtoboundary}
d_X(r_j(x),\partial\sigma) \leqslant d_X(x,\partial\sigma),\quad \forall x\in Sd^{j+1}\,\sigma,\; \forall j\geqslant 1.
\end{equation}
Note also that for $j=0$ the following holds trivially as both sides are zero
\begin{equation}\label{rjforzero}
d_X(r_0(x),\partial\sigma) \leqslant d_X(x,\partial\sigma),\quad \forall x\in Sd\,(\partial\sigma).
\end{equation}
Define $P_j:\id_X\simeq r_j$ to be the straight line homotopy for all $j\in\N$. By equation $(\ref{rjtoboundary})$, if $j\geqslant 1$ then  
\begin{equation}\label{Pjtoboundary}
d_X(P_j(x,t),\partial\sigma) \leqslant d_X(P_j(x,s),\partial\sigma),\quad \forall 0\leqslant s \leqslant t \leqslant 1,\;\forall x\in Sd^j\,\sigma.
\end{equation}
If $j=0$ this condition trivially holds for all $x\in Sd\,(\partial\sigma)$ but need not hold elsewhere.
\end{constr}

\begin{defn}
For all $i\leqslant j$, let $r_{i,i+1,\ldots,j}$ denote the composition \[r_i\circ r_{i+1}\circ \ldots \circ r_j\] and let $P_{i,i+1,\ldots,j}$ denote the concatenation of canonical straight line homotopies \[P_i(r_{i+1,\ldots,j})*\ldots * P_{j-1}(r_j)*P_j\]
\end{defn}

\begin{rmk}\label{Rmk:retractingrPareretracting}
Note that by construction \ref{constructr} $r_{0,\ldots,j-1}$ and $P_{0,\ldots,j-1}$ have the following properties
\begin{align}
 r_{0,\ldots,j-1}(Sd^j(\sigma) \backslash \mathring{D}(\widehat{\sigma},Sd^{j-1}\,\sigma)) &\subset \partial \sigma, \label{eleven} \\
 d_X(P_{0,\ldots, j-1}(x,t),\partial\sigma) &\leqslant d_X(P_{0,\ldots, j-1}(x,s),\partial\sigma), \label{twelve}
\end{align}
for all $\sigma\in X$, $0\leqslant s\leqslant t \leqslant 1$ and for all $x \in Sd^j\,\sigma \backslash \mathring{D}(\widehat{\sigma},Sd^{j-1}\,\sigma))$.
\end{rmk}

\begin{prop}\label{technicalprop}
Let $X$ be an f.d. l.f. simplicial complex, then for all $\ep<\comesh(X)$ there is an integer $i(X,\ep)$ such that for all integers $i\geqslant i(X,\ep)$, all $\sigma\in X$ and all $0\leqslant\ep^\prime \leqslant \ep$: 
\begin{eqnarray*}
 r_{0,\ldots,i-1}(N_\ep(Sd^i\, \sigma))&\subset& \sigma, \\ 
 P_{0,\ldots,i-1}(N_{\ep^\prime}(Sd^i\, \sigma),[0,1])&\subset& N_{\ep^\prime}(Sd^i\, \sigma).
\end{eqnarray*}
\end{prop}

\begin{proof}
Take any $\ep< \comesh(X)$. Let $i(X,\ep)$ be the smallest integer such that $\mesh(Sd^{i(X,\ep)}\,X) < \comesh(X) - \ep$. Note we can find such an integer since dim$(X)<\infty$ and \[\mesh(Sd^j\, X) \leqslant \left(\frac{\mathrm{dim}(X)}{\mathrm{dim}(X)+1}\right)^j\mesh(X).\] Hence for all $i\geqslant i(X,\ep)$, $\mesh(Sd^i\,X)< \comesh(X)-\ep$ so in particular 
\begin{equation}
D(\widehat{\sigma},Sd^{i-1}\,X) \subset Sd^i\,\sigma \backslash N_\ep(\partial\sigma) \label{inmiddle} 
\end{equation}
for all $\sigma\in X$.

The result now follows from $(\ref{inmiddle})$, $(\ref{eleven})$ and $(\ref{twelve})$ as for all $\tau \in X$, \[ N_\ep(\partial \tau) \subset Sd^i\,\tau\backslash D(\widehat{\tau},Sd^{i-1}\,\tau).\]
\end{proof}

\begin{cor}\label{Cor:lotsofeqns}
Let $\ep$, $i$ and $r=r_{0,\ldots, i-1}$ be given as in Proposition \ref{technicalprop}. Then the induced chain equivalence on locally finite simplicial chains 
\begin{displaymath}
\xymatrix@C=2cm{
(\Delta^{lf}_*(X),d_{\Delta^{lf}_*(X)},0) \ar@<0.5ex>[r]^-{s_*=\id_*} & (\Delta^{lf}_*(Sd^i\,X),d_{\Delta^{lf}_*(Sd^i\,X)},P_*)\ar@<0.5ex>[l]^-{r_*}
}
\end{displaymath}
satisfies
\begin{align}
 s_*(\Delta_*(X)(\sigma)) &\subset \Delta_*(Sd^i\, X)[Sd^i\,\sigma], \label{eye} \\
 r_*(\Delta_*(Sd^i\,X)[N_{\ep^\prime}(Sd^i\,\partial\sigma)])&\subset \Delta_*(X)[\partial\sigma], \label{eyeeye} \\
 P_*(\Delta_*(Sd^i\,X)[N_{\ep^\prime}(Sd^i\,\sigma)])&\subset \Delta_{*+1}(Sd^i\,X)[N_{\ep^\prime}(Sd^i\,\sigma)], \label{eyeeyeeye}
\end{align}
for all $\sigma\in X$ and all $0\leqslant \ep^\prime \leqslant\ep$.

The dual chain equivalence on simplicial cochains
\begin{displaymath}
\xymatrix@C=2cm{
(\Delta^{-*}(X),\delta^{\Delta^{-*}(X)},0) \ar@<0.5ex>[r]^-{r^*} & (\Delta^{-*}(Sd^i\,X),\delta^{\Delta^{-*}(Sd^i\,X)},P^*)\ar@<0.5ex>[l]^-{s^*}
}
\end{displaymath}
satisfies the dual conditions
\begin{align}
 s^*(\Delta^{-*}(Sd^i\,X)[Sd^i\,\mathring{\sigma}]) &\subset \Delta^{-*}(X)[\bigcup_{\tau\geqslant\sigma}\mathring{\tau}], \label{eyevee} \\
 r^*(\Delta^{-*}(X)(\sigma))&\subset \Delta^{-*}(X)[\bigcup_{\tau\geqslant\sigma}(Sd^i\,\tau\backslash N_\ep(\partial\tau\backslash\sigma))], \label{vee} \\
 P^*(\Delta^{-*}(Sd^i\,X)[Sd^i\,\rho\backslash N_{\ep^\prime}(\partial\rho\backslash\sigma)])&\subset \Delta^{-*+1}(Sd^i\,X)[Sd^i\,\rho\backslash N_{\ep^\prime/2}(\partial\rho\backslash\sigma)], \label{veeeye}
\end{align}
for all $\sigma\in X$, $0\leqslant\ep^\prime\leqslant\ep$ and $\rho\geqslant\sigma$.
\end{cor}

\begin{proof}
Statements $(\ref{eye})-(\ref{eyevee})$ follow directly from Proposition \ref{technicalprop}. For $(\ref{vee})$ note that $r_*$ sends $N_{\ep^\prime}(\partial\sigma)$ to $\partial\sigma$ and the support of $r^*(\mathring{\sigma})$ is all simplices in $Sd^i\,X$ sent by $r$ to $\mathring{\sigma}$, i.e. excluding the boundary of $\sigma$. For $(\ref{veeeye})$ note that statement $(\ref{eyeeyeeye})$ is saying that $P_*$ maps simplices of $Sd^i\,X$ \textit{towards the boundary} of the simplex in $X$ that they are contained in. Consider $P^*$ applied to the region $Y:=Sd^i\,\rho\backslash N_{\ep^\prime}(\partial\rho\backslash\sigma)$. The support of $P^*(Y)$ is the set of points $x$ whose paths $P(x,I)$ intersect $Y$. All points of $\rho\backslash Y$ that are nearer to $\partial\rho\backslash\sigma$ than they are to $Y$ must have paths disjoint from $Y$. This is certainly true for $Sd^i\,\rho\backslash N_{\ep^\prime/2}(\partial\rho\backslash\sigma)$ so the result follows.
\end{proof}

\begin{cor}\label{rscd}
Let $C\in\BB(\A^*(X))$ and let $r=r_{0,\ldots,i-1}$. Then the chain equivalence of Proposition \ref{chainequivforcompositions} 
\begin{displaymath}
\xymatrix@C=2cm{
(C,d_C,0) \ar@<0.5ex>[r]^-{s_C:=s_*} & (Sd_r\,C,d_{Sd_r\,C},P_C:=P_*)\ar@<0.5ex>[l]^-{r_C:=r_*}
}
\end{displaymath}
obtained by composing chain equivalences from Theorem \ref{Thm:subdivassemblechainequiv} satisfies
\begin{align}
 s_C(C(\sigma))&\subset (Sd_r\,C)[Sd^i\,\sigma], \label{seven} \\
 r_C((Sd_r\,C)[N_{\ep^\prime}(Sd^i\,\sigma)])&\subset C[\sigma], \label{eight} \\
 P_C((Sd_r\,C)_*[N_{\ep^\prime}(Sd^i\,\sigma)])&\subset (Sd_r\,C)_{*+1}[N_{\ep^\prime}(Sd^i\,\sigma)], \label{nine}
\end{align}
for all $\sigma\in X$ and $0\leqslant\ep^\prime\leqslant\ep$.

Let $C\in\BB(\A_*(X))$. Then the corresponding chain equivalence 
\begin{displaymath}
\xymatrix@C=2cm{
(C,d_C,0) \ar@<0.5ex>[r]^-{r^C:=r^*} & (Sd_r\,C,d_{Sd_r\,C},P^C:=P^*)\ar@<0.5ex>[l]^-{s^C:=s^*}
}
\end{displaymath}
satisfies 
\begin{align}
 s^C((Sd^i_r\,C)[Sd^i\,\mathring{\sigma}])&\subset C[\bigcup_{\tau\geqslant\sigma}\mathring{\tau}], \label{sCeqn} \\
 r^C(C(\sigma))&\subset (Sd^i_r\,C)[\bigcup_{\tau\geqslant\sigma}Sd^i\,\tau\backslash N_\ep(\partial\tau\backslash\sigma)], \label{rCeqn} \\
 P^C((Sd^i_r\,C)_*[Sd^i\,\rho\backslash N_{\ep^\prime}(\partial\rho\backslash\sigma)])&\subset (Sd^i_r\,C)_{*+1}[Sd^i\,\rho\backslash N_{\ep^\prime/2}(\partial\rho\backslash\sigma)], \label{PCeqn}
\end{align}
for all $\sigma\in X$, $\rho\geqslant\sigma$ and $0\leqslant\ep^\prime\leqslant\ep$.
\end{cor}

\begin{thm}[Squeezing Theorem]\label{triangulise}
Let $X$ be a finite-dimensional locally finite simplicial complex. There exists an $\ep=\ep(X)>0$ and an integer $i(X,\ep)$ such that for all $i\geqslant i(X,\ep)$ setting $r=r_{0,\ldots, i-1}$ there exists a chain equivalence $Sd_r\,C \to Sd_r\, D$ in $\mathbb{G}_{Sd^i\, X}(\A)$ with control at most $\ep$ measured in $X$ for $C,D\in \A(X)$, then there exists a chain equivalence $\xymatrix{ f: C \ar[r]^-{\sim} & D}$ in $\A(X)$ without subdividing.
\end{thm}

\begin{proof}
Let $\ep^\prime(X)$ and $i(X,\ep)$ be chosen as in Proposition \ref{technicalprop} and its corollaries. Set $\ep(X) = \frac{1}{5}\ep^\prime(X)$ and let $i\geqslant i(X,\ep)$. First suppose that $C,D\in \BB(\A^*(X))$ and that there exists a chain equivalence 
\begin{displaymath}
 \xymatrix{ (Sd_r\, C,d_{Sd_r\,C},Q_C) \ar@<0.5ex>[r]^{f^i} & (Sd_r\, D,d_{Sd_r\,D},Q_D) \ar@<0.5ex>[l]^{g^i}
}
\end{displaymath}
with control $\ep$. The following composition is also a chain equivalence: 
\begin{displaymath}
 \xymatrix{ (C,d_C,r_C(Q_C + g^i P_D f^i)s_C) \ar@<0.5ex>[rr]^-{r_D f^i \circ s_C} && (D,d_D,r_D(Q_D + f^i P_C g^i)s_D). \ar@<0.5ex>[ll]^-{r_C g^i \circ s_D}
}
\end{displaymath}
Examining this chain equivalence carefully we observe that it is in fact a chain equivalence in $\A^*(X)$:

Consider $r_Df^is_C$ applied to $C(\sigma)$ which we think of as supported on $\mathring{\sigma}$. By $(\ref{seven})$, $s_C(C(\sigma))$ is supported on $Sd^i\,\sigma$. Since $f^i$ has bound $\ep$, we see that $f^is_C(C(\sigma))$ is supported on $N_\ep(Sd^i\,\sigma)$. By $(\ref{eight})$, $r_Df^is_C(C(\sigma))$ is supported on $\sigma$, thus $r_Df^is_C$ is a morphism in $\A^*(X)$. For brevity in the following analyses we call this reasoning \textit{arguing by supports} and write
\begin{displaymath}
 \xymatrix{ r_Df^is_C: \mathring{\sigma} \ar[r]^-{(\ref{seven})} & Sd^i\, \sigma \ar[r] & N_\ep(Sd^i\, \sigma) \ar[r]^-{(\ref{eight})} & \sigma.
}
\end{displaymath}
See \Figref{Fig:Supports} for an example of this argument for a $2$-simplex.
\begin{figure}[h!]
\begin{center}
{
\psfrag{sstar}{$s_C$}
\psfrag{tstar}{$r_D$}
\psfrag{fi}{$f^i$}
\includegraphics[width=11cm]{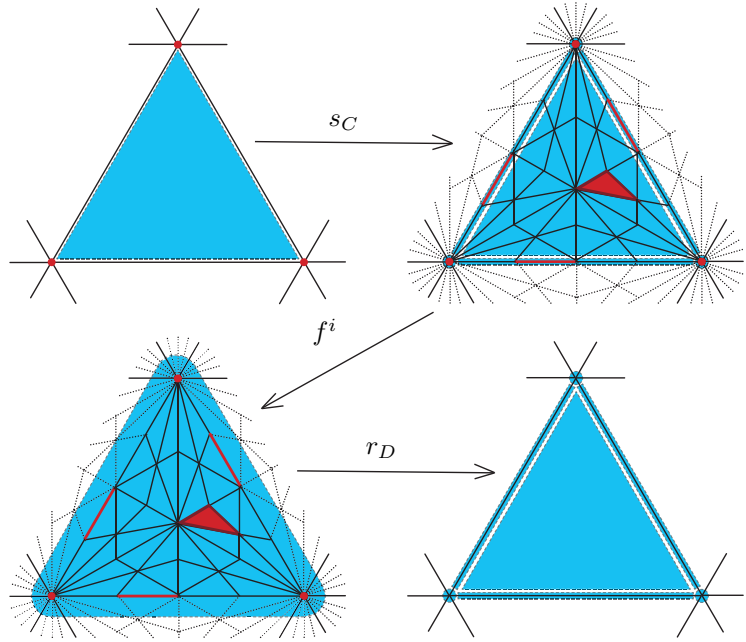}
}
\caption{Arguing by supports to show $r_Df^is_C\in \A^*(X)$.}
\label{Fig:Supports}
\end{center}
\end{figure}

By exactly the same argument we see that $r_Cg^is_D$, $r_CQ_Cs_C$ and $r_DQ_Ds_D$ are all morphisms in $\A^*(X)$, noting that $g^i$, $Q_C$ and $Q_D$ all have bound $\ep$. This just leaves $r_C(g^i(P_D)_*f^i)s_C$ and $r_D(f^i(P_C)_*g^i)s_D$ to check. Arguing by supports both of these send 
\begin{displaymath}
 \xymatrix{ \mathring{\sigma} \ar[r]^-{(\ref{seven})} & Sd^i\, \sigma \ar[r] & N_\ep(Sd^i\, \sigma) \ar[r]^-{(\ref{nine})} & N_\ep(Sd^i\, \sigma) \ar[r] & N_{2\ep}(Sd^i\, \sigma) \ar[r]^-{(\ref{eight})} & \sigma,
}
\end{displaymath}
so they are also morphisms of $\A^*(X)$, thus $r_Df^is_C: C\to D$ is a chain equivalence in $\A^*(X)$ as required.

Next suppose that $C,D\in \BB(\A_*(X))$ and that there exists a chain equivalence 
\begin{displaymath}
 \xymatrix{ (Sd_r\, C,d_{Sd_r\,C},Q_C) \ar@<0.5ex>[r]^{f^i} & (Sd_r\, D,d_{Sd_r\,D},Q_D) \ar@<0.5ex>[l]^{g^i}
}
\end{displaymath}
with control $\ep$. Again the following composition is a chain equivalence which we observe to be a chain equivalence in $\A_*(X)$: 
\begin{displaymath}
 \xymatrix{ (C,d_C,s^C(Q_C + g^iP^Df^i)r^C) \ar@<0.5ex>[rr]^-{s^Df^ir^C} && (D,d_D,s^D(Q_D + f^iP^Cg^i)r^D). \ar@<0.5ex>[ll]^-{s^Cg^ir^D}
}
\end{displaymath}
All of $s^Df^ir^C$, $s^Cg^ir^D$, $s^CQ_Cr^C$ and $s^DQ_Dr^D$ are morphisms in $\A_*(X)$ as they send
\begin{displaymath}
 \xymatrix{ \mathring{\sigma} \ar[r]^-{(\ref{rCeqn})} & \bigcup_{\tau\geqslant\sigma}(Sd^i\,\tau \backslash N_{5\ep}(\partial\tau\backslash \sigma)) \ar[r] & \bigcup_{\tau\geqslant\sigma}(Sd^i\,\tau \backslash N_{4\ep}(\partial\tau\backslash \sigma)) \ar[r]^-{(\ref{sCeqn})} & \bigcup_{\tau\geqslant \sigma}\mathring{\tau}. 
}
\end{displaymath}
Similarly, both  $s^C(g^iP^Df^i)r^C$ and $s^D(f^iP^Cg^i)r^D$ send 
\begin{displaymath}
 \xymatrix@R=3mm@C=7.8mm{ \mathring{\sigma} \ar[r]^-{(\ref{rCeqn})} & \bigcup_{\tau\geqslant\sigma}(Sd^i\,\tau \backslash N_{5\ep}(\partial\tau\backslash \sigma)) \ar[r] & \bigcup_{\tau\geqslant\sigma}(Sd^i\,\tau \backslash N_{4\ep}(\partial\tau\backslash \sigma)) \ar[r]^-{(\ref{PCeqn})} & \bigcup_{\tau\geqslant\sigma}(Sd^i\,\tau \backslash N_{2\ep}(\partial\tau\backslash \sigma)) \\
&\ar[r] & \bigcup_{\tau\geqslant\sigma}(Sd^i\,\tau \backslash N_{\ep}(\partial\tau\backslash \sigma)) \ar[r]^-{(\ref{sCeqn})} & \bigcup_{\tau\geqslant \sigma}\mathring{\tau}, 
}
\end{displaymath}
so are also morphisms in $\A_*(X)$. See \Figref{Fig:Supports2} for an example of this argument for a $1$-simplex.
\begin{figure}[ht]
\begin{center}
{
\psfrag{t}{}
\psfrag{A}[t][t]{$\mathring{\sigma}$}
\psfrag{B}[t][t]{$\;\bigcup_{\tau\geqslant\sigma}(Sd^i\,\tau \backslash N_{5\ep}(\partial\tau\backslash \sigma))$}
\psfrag{C}[t][t]{$\bigcup_{\tau\geqslant\sigma}(Sd^i\,\tau \backslash N_{\ep}(\partial\tau\backslash \sigma))$}
\psfrag{D}[t][t]{$\bigcup_{\tau\geqslant \sigma}\mathring{\tau}$}
\psfrag{a}[t][t]{$r^C$}
\psfrag{b}[tr][tr]{$g^iP^Df^i\quad\quad$}
\psfrag{c}[t][t]{$s^C$}
\includegraphics[width=9cm]{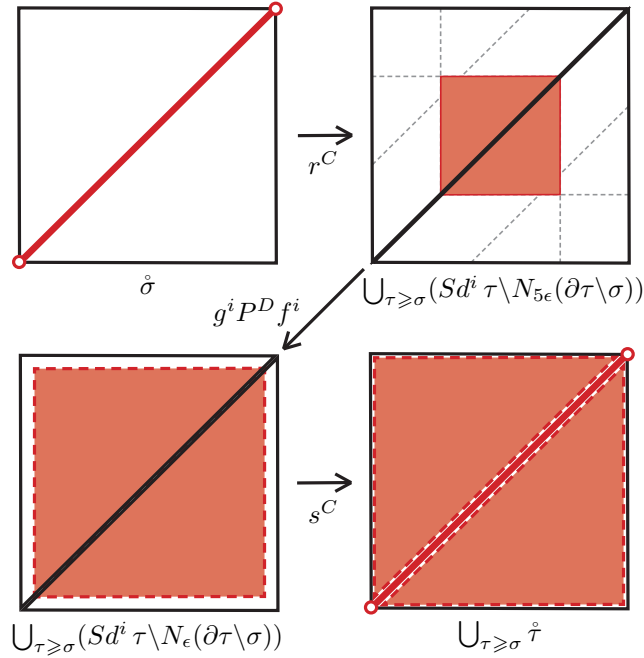}
}
\caption{Arguing by supports to show $s^Cg^iP^Df^ir^C\in \A_*(X)$.}
\label{Fig:Supports2}
\end{center}
\end{figure}
Thus $s^Df^ir^C:C\to D$ is a chain equivalence in $\A_*(X)$ as required.
\end{proof}

\section{Codimension one splitting over the open cone}\label{section:openconeandsplitting}
We now consider the following algebraic splitting problem. Given a chain complex $D\in \BB(\A(X\times \R))$ when can we find a chain complex $C\in \brc{\BB(\A^*(X))}{\BB(\A_*(X))}$ such that \[ D \simeq \brc{C \otimes \Delta^{lf}_*(\R)}{C \otimes \Delta^{-*}(\R)}?\] To answer this we must triangulate $X\times \R$ and decide what $\otimes$ means in this situation.

We construct a $\Z$-parameter family of triangulations of $X\times\R$ which have a finite mesh when measured in $O(X_+)$ with the coning map $j_X:X\times\R \to O(X_+)$.

\begin{constr}\label{buildt1xtimesR}
Let $X$ be an $(n-1)$-dimensional locally finite simplicial complex. Define a set of points $\{v_i\}_{i\in\Z}$ in $\R$ by
\[ v_i := \brcc{i,}{i\leqslant 0,}{\sum_{j=0}^{i-1} \left(\dfrac{n+2}{n+1}\right)^j,}{i>0.}\]
\end{constr}
Note that for all $i>0$:
\begin{equation}\label{scalefactor}
v_i - v_{i-1} = \dfrac{n+1}{n+2}(v_{i+1} - v_i).
\end{equation}
Fix a triangulation of $X\times I$ once and for all which we write as $\mathrm{Prism}(X,X)$. Following Remark \ref{prismremark} and Example \ref{prismexample} this gives a triangulation of $\mathrm{Prism}(X,Sd\,X)$ that is an iterated stellar subdivision of $\mathrm{Prism}(X,X)$ which we further use to set \[\mathrm{Prism}(Sd^i\, X,Sd^{i+1}\, X):= Sd^i\,\mathrm{Prism}(X,Sd\,X).\] For all $j\in \Z$ let $t^j(X\times \R)$ denote the triangulation of $X\times\R$ where $X\times [v_i,v_{i+1}]$ is given the triangulation $\mathrm{Prism}(Sd^{\max\{i-j,0\}}\,X,Sd^{\max\{i-j+1,0\}}\,X)$.

\begin{figure}[h!]
\begin{center}
{
\psfrag{2.5}{$\dfrac{7}{3}$}
\psfrag{1}{$1$}
\psfrag{0}{$0$}
\psfrag{-1}{$-1$}
\psfrag{t0}{$t^0(X\times\R)$}
\psfrag{t1}{$t^1(X\times\R)$}
\includegraphics[width=9cm]{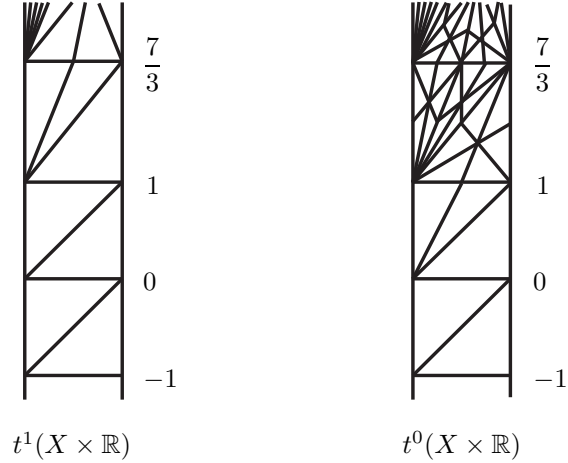}
}
\caption{Triangulating $[0,1]\times\R$.}
\label{Fig:tiXxR}
\end{center}
\end{figure}

\begin{rmk}\label{exptransassubdiv}
Note that for all $j\in \Z$, $k\geqslant 0$: $t^j(X\times\R)$ is an iterated stellar subdivision of $t^{j+k}(X\times\R)$ that is a $k$-fold iterated barycentric subdivision on $X\times [v_{j+k},\infty)$ and the identity subdivision on $X\times (-\infty,v_j]$.
\end{rmk}

To show that the triangulations $\{t^j(X\times\R)\}_{j\in\Z}$ just defined have finite mesh measured in $O(X_+)$ we need to study the metric on $O(X_+)$. In particular we have

\begin{prop}\label{annoyingconeestimates}
All simplices $\tau\in t^j(X\times [1,\infty))$ satisfy \[ \mesh_{O(X_+)}(Sd\,\tau) \leqslant \left( 1 - \dfrac{1}{(|\tau| +1)^3}\right) \diam_{O(X_+)}(\tau).\]
\end{prop}
\begin{proof}
Recall that for a simplex $\tau$ linearly embedded in Euclidean space we have \[\mesh(\tau) \leqslant \dfrac{|\tau|}{|\tau|+1}\diam(\tau).\] The proof of this fact in \cite{hatcher} may be adapted almost verbatim to the metric on $O(X_+)$ to give the desired result. The key difference is that when measuring the straight line in $X\times\R$ from $(x,s)$ to $(y,t)$ in $O(X_+)$ the usual equality for Euclidean space need not hold:
\[ d_{O(X_+)}\left((x,s), \dfrac{1}{m+1}(x,s) + \dfrac{m}{m+1}(y,t)\right) \neq   \dfrac{m}{m+1}d_{O(X_+)}((x,s),(y,t)).\] Let $m\leqslant n=\dim(X\times\R)$. We obtain an inequality using the formula for the metric on $O(X_+)$ as follows. Suppose $s\leqslant t$. Then
\begin{align}
 d_{O(X_+)}\left((x,s), \dfrac{1}{m+1}(x,s) + \dfrac{m}{m+1}(y,t)\right) &= s d_{X}\left(x,\dfrac{1}{m+1}x + \dfrac{m}{m+1}y\right) \notag \\
 &\quad+ \left(\dfrac{1}{m+1}s + \dfrac{m}{m+1}t-s\right) \notag \\
 &= \dfrac{m}{m+1}(s d_X(x,y) + (t-s)) \notag \\
 &= \dfrac{m}{m+1}d_{O(X_+)}((x,s),(y,t)). \notag
\end{align}
Suppose instead that $s\geqslant t$. Then
\begin{align}
 d_{O(X_+)}\left((x,s), \dfrac{1}{m+1}(x,s) + \dfrac{m}{m+1}(y,t)\right) &= \left(\dfrac{s}{m+1} + \dfrac{mt}{m+1} \right)d_{X}\left(x,\dfrac{x}{m+1} + \dfrac{my}{m+1}\right) \notag \\
&\quad+ \left(s-\dfrac{s}{m+1} + \dfrac{mt}{m+1}\right) \notag \\
&= \frac{mt + s}{(m+1)t} \dfrac{mt}{m+1}d_X(x,y) + \dfrac{m}{m+1}(s-t) \notag \\
&\leqslant \frac{mt + s}{(m+1)t} \left( \dfrac{mt}{m+1}d_X(x,y) + \dfrac{m}{m+1}(s-t) \right)\notag \\
&= \left(1 + \dfrac{s-t}{(m+1)t}\right)\dfrac{m}{m+1}d_{O(X_+)}((x,s),(y,t)). \notag 
\end{align}
Since each simplex of $t^j(X\times [1,\infty))$ is contained in a block $X\times [v_i,v_{i+1}]$ for some $i\geqslant 1$ we have
\begin{align}
 s-t &\leqslant v_{i+1}-v_i = \left(\dfrac{n+2}{n+1}\right)^i \notag \\
 t &\geqslant v_i = \sum_{k=0}^{i-1} \left(\dfrac{n+2}{n+1}\right)^k = (n+1)\left( \left(\dfrac{n+2}{n+1}\right)^i -1\right). \label{needlatertechnical}
\end{align}
Hence 
\[\dfrac{s-t}{t} \leqslant \dfrac{\left(\dfrac{n+2}{n+1}\right)^i}{(n+1)\left(\left(\dfrac{n+2}{n+1}\right)^i -1 \right)}\leqslant \dfrac{1}{n+1} \left(1- \left(\dfrac{n+1}{n+2}\right)\right)^{-1}= \dfrac{n+2}{n+1} \]
and so
\begin{align}
d_{O(X_+)}\left((x,s), \dfrac{1}{m+1}(x,s) + \dfrac{m}{m+1}(y,t)\right) &\leqslant   \left(1+ \dfrac{n+2}{(m+1)(n+1)}\right)\dfrac{m}{m+1} d_{O(X_+)}((x,s),(y,t)) \notag \\
&\leqslant   \left(1+ \dfrac{m+1}{(m+1)^2}\right)\dfrac{m}{m+1} d_{O(X_+)}((x,s),(y,t)) \notag \\
&\leqslant \left(1 - \dfrac{1}{(m+1)^3}\right)d_{O(X_+)}((x,s),(y,t))\notag
\end{align}
from which the result follows.
\end{proof}

\begin{prop}
For all $j\in \Z$: $\mesh_{O(X_+)}(t^j(X\times\R)) < \infty$.
\end{prop}
\begin{proof}
For all $j\in \Z$ it is easy to bound the mesh of the lower blocks:
\begin{align}
 \mesh_{O(X_+)}(t^j(X\times (-\infty,0])) &= 1 < \infty\notag \\
 \mesh_{O(X_+)}(t^j(X\times [0,1])) &\leqslant \mesh_X(X) + 1 < \infty. \notag 
\end{align}
By Proposition \ref{annoyingconeestimates} it suffices to show that $\mesh_{O(X_+)}(t^j(X\times [1,\infty)))<\infty$ for $j\geqslant 1$ as $t^j(X\times\R)$ is a subdivision of $t^{j+k}(X\times\R)$ for $k\geqslant 0$. 

So let $j\geqslant 0$ and consider $t^j(X\times [v_j,v_{j+1}])$. This block is triangulated with $\mathrm{Prism}(X,Sd\,X)$. We estimate the maximum size a simplex can have here when measured in $O(X_+)$:
\begin{align}
 \mesh_{O(X_+)}(t^j(X\times [v_j,v_{j+1}])) &\leqslant v_{j+1}\mesh_X(X) + (v_{j+1}-v_j) \notag \\
 &= v_{j+1}\sqrt{2} + \left(\dfrac{n+2}{n+1}\right)^j=: B(j)<\infty. \notag
\end{align}
Lower blocks necessarily have a mesh bounded by $B(j)$ so we just have to check higher blocks. By construction 
\[t^j(X\times [v_{j+k},v_{j+k+1}]) = Sd^k\, t^{j+k}(X\times [v_{j+k},v_{j+k+1}]),\] so by Proposition \ref{annoyingconeestimates}:
\begin{align}
\mesh_{O(X_+)}(t^j(X\times [v_{j+k},v_{j+k+1}])) &\leqslant ( 1 - \dfrac{1}{(n+1)^3})^k \mesh_{O(X_+)}(t^{j+k}(X\times [v_{j+k},v_{j+k+1}])) \notag \\
&= ( 1 - \dfrac{1}{(n+1)^3})^k B(j+k). \notag 
\end{align}
One can easily check that \[ \molim_{k\to \infty}( 1 - \dfrac{1}{(n+1)^3})^k B(j+k) < \infty,\]so we are done.
\end{proof}

\begin{cor}
Let $n = \dim(X)+1$. Then for all $j\in \Z$, $k\in\N$,  \[\mesh_{O(X_+)}(t^j(X\times [v_i,\infty))) \leqslant ( 1 - \dfrac{1}{(n+1)^3})^k \mesh_{O(X_+)}(t^{j+k}(X\times [v_i,\infty)))\] where $v_i = \max\{1,v_{j+k}\}$. 
\end{cor}

\begin{prop}\label{exptransscalesthings}
For all $z_1$, $z_2\in X\times [1,\infty)$, \[ d_{O(X_+)}(t^{-1}(z_1),t^{-1}(z_2)) \leqslant \dfrac{n+1}{n+2}d_{O(X_+)}(z_1,z_2).\]
\end{prop}
\begin{proof}
Let $z_1\in X\times [v_i,v_{i+1}]$ and $z_2\in X\times [v_j,v_{j+1}]$. Then
\begin{align}
 z_1 &= \left(x_1,v_i + s_1\left(\dfrac{n+2}{n+1}\right)^i \right), \notag \\
 z_2 &= \left(x_2,v_j + s_2\left(\dfrac{n+2}{n+1}\right)^j \right), \notag
\end{align}
for some $x_1$, $x_2\in X$, $s_1$, $s_2\in [0,1]$. Without loss of generality let's assume \[v_i + s_1\left(\dfrac{n+2}{n+1}\right)^i \leqslant v_j + s_2\left(\dfrac{n+2}{n+1}\right)^j\] so certainly we have $1\leqslant i\leqslant j$. Exponentially translating we get 
\begin{align}
 t^{-1}(z_1) &= \left(x_1,v_{i-1} + s_1\left(\dfrac{n+2}{n+1}\right)^{i-1} \right), \notag \\
 t^{-1}(z_2) &= \left(x_2,v_{j-1} + s_2\left(\dfrac{n+2}{n+1}\right)^{j-1} \right), \notag
\end{align}
From $(\ref{needlatertechnical})$ we deduce that $v_{i-1} \leqslant \dfrac{n+1}{n+2}v_i$. Hence
\begin{align}
 d_{O(X_+)}(t^{-1}(z_1),t^{-1}(z_2)) &= \left( v_{i-1} + s_1\left(\dfrac{n+2}{n+1}\right)^{i-1} \right)d_X(x_1,x_2) + v_{j-1}-v_{i-1} \notag \\ 
 &\quad + s_2\left(\dfrac{n+2}{n+1}\right)^{j-1} - s_1\left(\dfrac{n+2}{n+1}\right)^{i-1} \notag \\
 &\leqslant \dfrac{n+1}{n+2} \left(\left( v_{i} + s_1\left(\dfrac{n+2}{n+1}\right)^{i} \right)d_X(x_1,x_2) + v_{j}-v_{i}\right. \notag \\
 &\quad \left. + s_2\left(\dfrac{n+2}{n+1}\right)^{j} - s_1\left(\dfrac{n+2}{n+1}\right)^{i} \right) \notag \\
 &= \dfrac{n+1}{n+2}d_{O(X_+)}(z_1,z_2),\notag 
\end{align}
where we have used $(\ref{scalefactor})$ to deduce that
\[v_{j-1}-v_{i-1} = \dfrac{n+1}{n+2}(v_j-v_i).\]
\end{proof}

\begin{defn}
For all $k\in \Z$, define a PL isomorphism $t^k:X\times\R \to X\times\R$, called \textit{exponential translation by $k$}, by 
\[\begin{array}{rcl} t^k:X\times\R &\to& X\times\R \\ s(x,v_i)+(1-s)(x,v_{i+1}) &\mapsto& s(x,v_{i+k})+(1-s)(x,v_{i+1+k}) \end{array}\] for all $s\in[0,1]$ and $i\in\Z$. Note that $t^k$ maps the triangulation $t^j(X\times\R)$ to $t^{j+k}(X\times\R)$ for all $j,k\in \Z$.
\end{defn}

By Remark \ref{exptransassubdiv} the exponential translation maps $\{t^k\}_{k\in\Z}$ can be thought of as subdivisions for $k<0$ and assemblies for $k>0$. They induce isomorphisms of categories:
\[(t^k)_*: \left\{ \begin{array}{c} \A^*(t^j(X\times\R)) \to \A^*(t^{j+k}(X\times\R)) \\ \A_*(t^j(X\times\R)) \to \A_*(t^{j+k}(X\times\R)) \\ \mathbb{G}_{t^j(X\times\R)}(\A) \to \mathbb{G}_{t^{j+k}(X\times\R)}(\A) \end{array}\right. \]
for all $j,k\in \Z$.

\begin{defn}
For $k>0$, exponential translation by $-k$ thought of as an iterated stellar subdivision induces algebraic subdivision functors as in Definition \ref{disjointstellars} \[Sd_k:\BB(\A(t^j(X\times\R))) \to \BB(\A(t^{j-k}(X\times\R))).\] 
In general $Sd_k$ is very different to $(t^{-k})_*$ but if $(t^{-1})_*(C) \simeq Sd_1\,C\in \BB(\A(t^{j-1}(X\times\R)))$ then we say that the chain complex $C\in \BB(\A(t^j(X\times\R)))$ is \textit{exponential translation equivalent}.\footnote{Or equivalently if $t_*Sd_1\,C \simeq C$.}
\end{defn}

\begin{ex}
Let $C=\brc{\Delta^{lf}_*(t^j(X\times\R))}{\Delta^{-*}(t^j(X\times\R)).}$ Then $(t^{-1})_*C = Sd_1\, C$ by Proposition \ref{alggeneralisestop}, so $C$ is certainly exponential translation equivalent.
\end{ex}

The triangulations $t^j(X\times\R)$ and the exponential translation maps $t^k$ have been carefully constructed to have the following key properties:
\begin{rmk}\label{slide}
\begin{enumerate}[(i)]
 \item Measuring in $O(X_+)$ with the coning map $\mesh(t^j(X\times\R))<\infty$.
 \item If $f:C\to D\in \BB(\GG_{t^j(X\times\R)})$ has bound $B<\infty$ measured in $O(X_+)$, then \[(t^{-1})_*f:(t^{-1})_*C \to (t^{-1})_*D \in \BB(\GG_{t^{j-1}(X\times\R)})\] has bound $\dfrac{n+1}{n+2}B$ when restricted to $t^{j-1}(X\times[0,\infty))$ measured in $O(X_+)$.
\end{enumerate}
\end{rmk}
Thus exponential translation $(t^{-1})_*$ allows us to rescale the bound of a chain map in $\BB(\GG_{t^j(X\times\R)})$ until it is small enough to apply the squeezing theorem. Taking this approach with exponential translation equivalent chain complexes yields the following:

\begin{thm}[Splitting Theorem]\label{crucsplit}
Let $C,D$ be exponential translation equivalent chain complexes in $\BB(\A(t^1(X\times\R)))$. If there exists a chain equivalence $f:C\to D$ in $\mathbb{G}_{t^1(X\times\R)}(\A)$ with finite bound $0\leqslant B<\infty$ measured in $O(X_+)$, then for all $i\geqslant 1$ there exists a chain equivalence \[ f_i: C|_{t^1(X\times\{v_i\})} \to D|_{t^1(X\times\{v_i\})}\] in $\A(t^1(X\times\{v_i\})).$ Projecting to $X\times\{1\}$ this is a chain equivalence in $\A(Sd^{i-1}\,X)$ with bound tending to zero as $i\to \infty.$
\end{thm}

\begin{proof}
If $B=0$, just take $f_i=f|: C|_{t^1(X\times\{v_i\})} \to D|_{t^1(X\times\{v_i\})}$. So suppose $B>0$. For all $v_j>B+1$ we can find an interval $J:=[v_{j_-},v_{j_+}]\subset [1,\infty)$ containing $(v_j-B, v_j+B)$. Since $\comesh(X)>0$, $\comesh(\mathrm{Prism}(X,Sd\,X))>0$ and $\comesh(t^1(X\times J))>0$ so we may apply the Squeezing Theorem to it. Let \[\ep = \ep(t^1(X\times J)),\quad i = i(t^1(X\times J))\] be as given by the Squeezing Theorem.

By Remark \ref{slide} there exists a $k\geqslant i$ such that $(t^{-k})_*f: (t^{-k})_*C\to (t^{-k})_*D$ is a chain equivalence in $\mathbb{G}_{t^{1-k}(X\times\R)}(\A)$ with bound less than $\frac{\ep}{3}$. Exponential translation equivalence provides chain equivalences
\begin{displaymath}
\xymatrix@R=1mm{
\phi_C: (t^{-k})_*C \ar[r]^-{\sim} & Sd_kC, \\
\phi_D: (t^{-k})_*D \ar[r]^-{\sim} & Sd_kD,
}
\end{displaymath}
in $\A(t^{1-k}(X\times\R))$ with control at most $\mesh(t^{1-k}(X\times\R))<\ep/3$, where $Sd_k$ is the subdivision functor obtained from viewing $t^{1-k}(X\times\R)$ as a subdivision of $t^1(X\times\R)$. The composition $\widetilde{f}:= \phi_D\circ (t^{-k})_*f\circ\phi_C^{-1}$ is a chain equivalence in $\mathbb{G}_{t^{1-k}(X\times\R)}(\A)$ with control $<\ep$.

Consider the composition \[f_i:=\brc{r_D\widetilde{f} s_C}{s^D\widetilde{f}r^C}:C\to D,\]where $s_C,$ $r_D$, $s^D$ and $r^C$ are as in Corollary \ref{rscd} but with respect to the subdivision functor $Sd_k$. By the proof of the Squeezing theorem and the fact that $Sd_k = Sd^k$ on $\BB(\A(t^1(X\times J)))$, the restriction to $t^1(X\times\{v_i\})$ of $f_i$, it's chain inverse and the chain homotopies are all morphisms of $\A(t^1(X\times\{v_i\}))$ so we get \[f_i:C|_{t^1(X\times\{v_i\})}\to D|_{t^1(X\times\{v_i\})}\] a chain equivalence in $\BB(\A(t^1(X\times\{v_i\})))$. Exponential translation equivalence of $C$ and $D$ plus Lemma \ref{Lem:SubcontIsCont} give the desired $f_j$ for all $j<i$.
\end{proof}

\begin{defn}\label{algcrosswithR}
Define a functor \[\textit{``}-\otimes\Z\textit{''}: \A(X) \to \A(t^1(X\times\{v_i\}_{i\in\Z}))\subset\A(t^1(X\times\R))\] by sending an object $M$ of $\A(X)$ to the object of $\A(t^1(X\times\{v_i\}_{i\in\Z}))$ that is $Sd^j\,M$ on $X\times\{v_i\}$ for $j=\max\{ 0,i-1\}$ and by sending a morphism $f:M\to N$ of $\A(X)$ to the morphism of $\A(t^1(X\times\{v_i\}_{i\in\Z}))$ that is $Sd_{r_{0,\ldots,j-1}}\,f:Sd^j\, M \to Sd^j\, N$ on $X\times\{v_i\}$ again for $j=\max\{0,i-1\}$. As before this also defines a functor \[\textit{``}-\otimes\Z\textit{''}: \BB(\A(X)) \to \BB(\A(t^1(X\times\{v_i\}_{i\in\Z}))).\] 
\qed\end{defn}

\begin{ex}\label{Gexptrans}
For $C$ a chain complex in $\BB(\A(X))$, we have that $\textit{``}C\otimes\Z\textit{''}$ is exponential translation equivalent. This is almost a tautology from the fact that by definition \[t^{-i}(t^j(X\times\R)) = Sd_i(t^j(X\times\R))\] and $\textit{``}C\otimes\Z\textit{''}$ is defined to be $Sd_{r_{0,\ldots,i-1}}\, C$ on $t^1(X\times\{v_i\})$ where $j=\max\{0,i-1\}$.
\end{ex}

\begin{thm}\label{alganal}
Let $X$ be a finite-dimensional locally finite simplicial complex and let $C\in \BB(\A(X))$. Then the following are equivalent:
\begin{enumerate}
 \item $C(\sigma)\simeq 0 \in \A$ for all $\sigma \in X$,
 \item $C\simeq 0 \in \A(X)$,
 \item $\textit{``}C\otimes\Z\textit{''}\simeq 0\in\mathbb{G}_{t^1(X\times\R)}(\A)$ with finite bound measured in $O(X_+)$.
\end{enumerate}
\end{thm}
\begin{proof}
$(1)\Leftrightarrow (2)$: Proposition \ref{chcont}.

\noindent $(2)\Rightarrow(3)$: Immediate from the definitions and Lemma \ref{Lem:SubcontIsCont}.

\noindent $(3)\Rightarrow (2)$: Let $B$ be the bound of the chain contraction $\textit{``}C\otimes\Z\textit{''} \simeq 0$ in $\mathbb{G}_{t^1(X\times\R)}(\A)$ when measured in $O(X_+)$. Choose $\ep<\comesh(X)$ and let $i=i(X,\ep)$ be as in the Squeezing Theorem. We may choose $j$ large enough so that $v_j-v_{j-1}>B$, $B/v_j<\ep$ and $j>i$. The restriction of $\textit{``}C\otimes\Z\textit{''} \simeq 0$ to $X\times \{v_j\}$ projects to a chain equivalence $Sd^j\,C\simeq 0$ in $\mathbb{G}_{X}(\A)$ with bound at most $B/v_j<\ep$. Applying the Squeezing Theorem we get a chain equivalence $C\simeq 0$ in $\BB(\A(X))$.
\end{proof}

\begin{rmk}
Rather than defining a functor $\textit{``}-\otimes\Z\textit{''}$ one could instead define $\textit{``}-\otimes\R\textit{''}$. The way to do this is first to define
\[ \textit{``}-\otimes I\textit{''}: \BB(\A(X)) \to \BB(\A(\mathrm{Prism}(X,X)))\] which is used to define $\textit{``}-\otimes\R\textit{''}$ on lower blocks. Then compose this with a subdivision functor
\[\BB(\A(\mathrm{Prism}(X,X))) \to \BB(\A(\mathrm{Prism}(X,Sd\,X)))\] for the block $t^1(X\times[0,1])$ and then with further algebraic subdivision functors for higher blocks. Theorem \ref{alganal} still holds with $\textit{``}-\otimes\R\textit{''}$. The proof is still a relatively straight forward application of exponential translation invariance and the Splitting Theorem.
\end{rmk}

\section{Poincar\'{e} duality}\label{section:PD}
In this section let $R$ be a commutative ring and $\A=\F(R)$. Since the simplicial $\brc{\mathrm{chain}}{\mathrm{cochain}}$ complex $\brc{\Delta^{lf}_*(X)}{\Delta^{n-*}(X)}$ is naturally a chain complex in $\brc{\BB(\A^*(X))}{\BB(\A_*(X))}$ the results of this paper have applications to Poincar\'{e} duality.

\begin{defn}
 \begin{enumerate}[(i)]
  \item An \textit{$n$-dimensional $R$-homology Poincar\'{e} complex} is an f.d. l.f. simplicial complex together with a fundamental class $[X]\in \Delta^{lf}_n(X;R)$ such that the cap products \[ [X]\cap-: \Delta^{n-*}(X;R) \to \Delta^{lf}_*(Sd\,X;R)\] are chain equivalences.
  \item An $n$-dimensional $R$-homology Poincar\'{e} complex $X$ is called an \textit{$\ep$-controlled Poincar\'{e} complex} if there exists an $i\in \N$ such that viewing $\Delta^{n-*}(Sd^i\,X)$ and $\ttt \Delta^{lf}_*(Sd^{i+1}\, X)$ as chain complexes in $\BB(\A_*(Sd^i\,X))$ the assembled cap product maps
\[ [X]\cap-: \Delta^{n-*}(Sd^i\,X;R) \to \ttt \Delta^{lf}_*(Sd^{i+1}\,X;R)\] are chain equivalences in $\GG_{Sd^i\,X}(\A)$ with control at most $\ep$.
  \item An $n$-dimensional $R$-homology Poincar\'{e} complex is called a \textit{$\BB(\A_*(Sd^i\,X))$-controlled Poincar\'{e} complex} if the assembled cap product maps
\[ [X]\cap-: \Delta^{n-*}(Sd^i\,X;R) \to \ttt \Delta^{lf}_*(Sd^{i+1}\,X;R)\] are chain equivalences in $\BB(\A_*(Sd^i\,X))$.
 \item An $n$-dimensional $R$-homology Poincar\'{e} complex $X$ is an \textit{$n$-dimensional $R$-homology manifold} if for all $\sigma\in X$, 
\[\Delta_*(X,X\backslash\widehat{\sigma};R) \simeq \Sigma^n R.\]
 \item An f.d. l.f. simplicial complex $X$ is a combinatorial manifold if for all $\sigma\in X$, \[\link(\sigma,X)\cong S^{n-|\sigma|-1}.\]
 \item For an $n$-dimensional $R$-homology Poincar\'{e} complex $X$ we say that $X\times\R$ is a \textit{bounded $(n+1)$-dimensional Poincar\'{e} complex measured in $O(X_+)$} if there is a $B<\infty$ and a fundamental class $[t^1(X\times\R)]\in \Delta_{n+1}^{lf}(t^(X\times\R))$ such that the cap products
\[[t^1(X\times\R)]\cap -: \Delta^{n+1-*}(t^1(X\times\R)) \to \ttt\Delta^{lf}_*(Sd\,t^1(X\times\R)) \] are chain equivalences in $\GG_{t^1(X\times\R)}(\A)$ with bound at most $B$ measured in $O(X_+)$. Note here $Sd\, t^1(X\times\R)$ is the global barycentric subdivision of $t^1(X\times\R)$.
 \end{enumerate}
\end{defn}

\begin{rmk}
Note that, by Proposition \ref{chcont}, $X$ has $\BB(\A_*(Sd^i\,X))$-controlled Poincar\'{e} if and only if $(D(\sigma,X),\partial D(\sigma,X))$ is an $(n-|\sigma|)$-dimensional $R$-homology Poincar\'{e} pair.
\end{rmk}

This section is devoted to proving the following result

\begin{thm}\label{hommfldequivs}
Let $X$ be an $n$-dimensional $R$-homology Poincar\'{e} complex. Then the following are equivalent
\begin{enumerate}
 \item $X$ is an $n$-dimensional $R$-homology manifold.
 \item $X$ is an $\ep$-controlled Poincar\'{e} complex, for all $\ep>0$.
 \item $X$ is a $\BB(\A_*(Sd^i\,X))$-controlled Poincar\'{e} complex, for all $i\in\N$.
 \item $X$ is a $\BB(\A_*(X))$-controlled Poincar\'{e} complex.
 \item $X\times\R$ is a bounded Poincar\'{e} complex over $O(X_+)$.
\end{enumerate}
\end{thm}

\begin{proof}
$(1)\Leftrightarrow (4)$: Proposition \ref{characterisehommflds}.

$(3)\Rightarrow (4)$: Trivial.

$(4)\Rightarrow (3)$: Theorem \ref{subdividingtriangularPD}.

$(3)\Rightarrow (2)$: A consequence of the fact that $\mesh(Sd^i\,X)\to 0$ as $i\to\infty$.

$(2)\Rightarrow (4)$: The Poincar\'{e} duality Squeezing Theorem.

$(5)\Rightarrow (3)$: The Poincar\'{e} duality Splitting Theorem.

$(3)\Rightarrow (5)$: Apply $\textit{``}-\otimes\R\textit{''}$ to the chain equivalence \[\Delta^{n-*}(X)\simeq \ttt \Delta^{lf}_*(Sd\, X)\in \BB(\A_*(X)).\]
\end{proof}

The following is due to Ranicki.

\begin{prop}\label{characterisehommflds}
Let $X$ be an $n$-dimensional $R$-homology Poincar\'{e} complex. Then $X$ is a an $n$-dimensional $R$-homology manifold if and only if $X$ has $\BB(\A_*(X))$-controlled Poincar\'{e} duality.
\end{prop}

\begin{proof}
Observe that 
\begin{align}
 \Delta_*(X,X\backslash \widehat{\sigma};R) &\cong \Delta_*(\sigma * \link(\sigma,X), \sigma* \link(\sigma,X) \backslash \widehat{\sigma};R) \notag \\
 &\cong \Delta_*(\sigma * \partial D(\sigma,X), \partial\sigma * \link(\sigma,X);R) \notag \\
 &\cong \Delta_*(S^{|\sigma|-1} * D(\sigma,X), S^{|\sigma|-1} * \partial D(\sigma,X);R) \notag \\
 &\cong \Sigma^{|\sigma|}\Delta_*(D(\sigma,X),\partial D(\sigma,X);R), \notag
\end{align}
where we have used the fact that there is a simplicial isomorphism \[\partial D(\sigma,X) \cong Sd\, \link(\sigma, X).\]
Hence, $X$ is an $n$-dimensional $R$-homology manifold if and only if \[\Sigma^n\,R \cong \Sigma^{|\sigma|}\Delta_*(D(\sigma,X),\partial D(\sigma,X);R)\] if and only if \[\Delta_*(D(\sigma,X),\partial D(\sigma,X);R)\simeq \Sigma^{n-|\sigma|}R = \Delta^{n-*}(\sigma,\partial\sigma),\]i.e. if and only if $X$ has $\BB(\A_*(X))$-controlled Poincar\'{e} duality.
\end{proof}

\begin{rmk}
Since $(\sigma,\partial\sigma)$ is an $|\sigma|$-dimensional combinatorial manifold with boundary setting $\sigma = \tau_k$ we necessarily have \[\link(\widehat{\tau}_0\ldots\widehat{\tau}_k,Sd\, \tau_k)\cong S^{|\tau_k|-k-1}.\]
\end{rmk}

\begin{cor}
$\mathring{D}(\widehat{\tau}_0\ldots\widehat{\tau}_k,Sd\, \tau_k) \cong \mathring{\Delta}^{|\tau_k|-k}$
\end{cor}
\begin{proof}
\begin{align}
\mathring{D}(\widehat{\tau}_0\ldots\widehat{\tau}_k,Sd\, \tau_k) &=\mathrm{Int}(\widehat{\widehat{\tau}_0\ldots\widehat{\tau}_k} * \partial D(\widehat{\tau}_0\ldots\widehat{\tau}_k,Sd\,\tau_k))\notag \\
& \cong \mathrm{Int}(\mathrm{pt}* Sd\, \link(\widehat{\tau}_0\ldots\widehat{\tau}_k,Sd\,\tau_k,Sd\, \tau_k) )\notag \\
&\cong \mathrm{Int}(\mathrm{pt}* Sd\, S^{|\tau_k|-k-1} )\notag \\
&\cong \mathring{\Delta}^{|\tau_k|-k}.\notag
\end{align}
 \end{proof}

\begin{prop}\label{keypropforsubdivofPD}
There is a PL isomorphism \[\mathring{D}(\widehat{\tau}_0\ldots\widehat{\tau}_k,Sd\,X) \to \mathring{D}(\widehat{\tau}_0\ldots\widehat{\tau}_k,Sd\,\tau_k) * \partial D(\tau_k,X).\]
\end{prop}

\begin{proof}
A vertex $\widehat{\sigma}\in \mathring{D}(\widehat{\tau}_0\ldots\widehat{\tau}_k,Sd\,X)$ has $\sigma = \widehat{\sigma}_0\ldots \widehat{\sigma}_m$ for some $\sigma_0<\ldots < \sigma_m \subset Sd\, X$ with $\widehat{\tau}_0\ldots\widehat{\tau}_k<\widehat{\sigma}_0\ldots \widehat{\sigma}_m$. Thus there is an $l\geqslant k$ such that $\sigma_l = \tau_k$. Whence $\widehat{\tau}_0\ldots\widehat{\tau}_k\leqslant \widehat{\sigma}_0\ldots \widehat{\sigma}_l \in Sd\,\tau_k$ and $\tau_k<\sigma_{l+1}<\ldots < \sigma_m\subset Sd\, X$. In particular if $l<m$ we have $\widehat{\sigma}_{l+1}\ldots \widehat{\sigma}_m \in \partial D(\tau_k,X)$.

Define a PL map by \[\begin{array}{rcl}\Phi:\mathring{D}(\widehat{\tau}_0\ldots\widehat{\tau}_k,Sd\,X) &\to& \mathring{D}(\widehat{\tau}_0\ldots\widehat{\tau}_k,Sd\,\tau_k) * \partial D(\tau_k,X) \\ \widehat{\sigma} &\mapsto& \brCCC{\widehat{\sigma}\in \mathring{D}(\widehat{\tau}_0\ldots\widehat{\tau}_k,Sd\,\tau_k),}{l=m,}{\widehat{\sigma}_{k+1}\ldots \widehat{\sigma}_{m} \in \partial D(\tau_k,X),}{l=k,}{\dfrac{1}{2}(\widehat{\sigma}_{0}\ldots \widehat{\sigma}_{l})+\dfrac{1}{2}(\widehat{\sigma}_{l+1}\ldots \widehat{\sigma}_{m}),}{\mathrm{otherwise.}}  \end{array}\]
An elementary, yet lengthy calculation verifies this is indeed a PL isomorphism.
\end{proof}

\begin{cor}
Since \[\mathring{D}(\widehat{\tau}_0\ldots\widehat{\tau}_k,Sd\,\tau_k) * \partial D(\tau_k,X) \cong \partial D(\widehat{\tau}_0\ldots\widehat{\tau}_k,Sd\,\tau_k) * \mathring{D}(\tau_k,X)\] we get that 
\begin{equation}\label{keyequation}
\mathring{D}(\widehat{\tau}_0\ldots\widehat{\tau}_k,Sd\,X) \cong S^{|\tau_k|-k-1} * \mathring{D}(\tau_k,X). 
\end{equation}
\end{cor}

\begin{thm}\label{subdividingtriangularPD}
Let $X$ be an $n$-dimensional $R$-homology Poincar\'{e} complex. Then $X$ has $\BB(\A(X))$-controlled Poincar\'{e} duality if and only if $X$ has $\BB(\A(Sd\,X))$-controlled Poincar\'{e} duality.
\end{thm}

\begin{proof}
This is now a direct consequence of equation $(\ref{keyequation})$.
\end{proof}

\begin{thm}[Poincar\'{e} Duality Squeezing]\label{PDsqueezing}
Let $X$ be an $n$-dimensional $R$-homology Poincar\'{e} complex. There exists an $\ep=\ep(X)>0$ and an integer $i=i(X,\ep)$ such that for all $j\geqslant i$ if $X$ has an $\ep$-controlled Poincar\'{e} duality chain equivalence \[ [X]\cap -: \Delta^{n-*}(Sd^i\, X) \to \ttt_i\Delta_*^{lf}(Sd^{i+1} X), \]then $X$ has $\BB(\A_*(X))$-controlled Poincar\'{e} duality.
\end{thm}

\begin{proof}
Let $\ep=\ep(Sd\,X)$ and $i=i(Sd\,X,\ep)$ be as in Theorem \ref{triangulise}. Let $j\geqslant i$ and suppose that \[ [X]\cap-: \Delta^{n-*}(Sd^j\,X;R) \to \ttt \Delta^{lf}_*(Sd^{j+1}\,X;R)\] are chain equivalences in $\GG_{Sd^j\,X}(\A)$ with control at most $\ep$. By Proposition \ref{alggeneralisestop}
\begin{align}
\Delta^{n-*}(Sd^j\,X) &= Sd_{r_X}\,\Delta^{n-*}(X), \notag \\\Delta_*^{lf}(Sd^{j+1}\,X) &= Sd_{r_{Sd\,X}}\,\Delta^{lf}_*(Sd\,X),\notag 
\end{align}
for functors 
\begin{align}
& Sd_{r_X}:\BB(\A_*(X)) \to \BB(\A_*(Sd^j\,X)), \notag \\
& Sd_{r_{Sd\,X}}:\BB(\A^*(Sd\,X)) \to \BB(\A^*(Sd^{j+1}\,X)) \notag
\end{align}
defined using any valid choice of simplicial approximations to the identity  
\begin{align}
& r_X:Sd^j\,X \to X, \notag \\
& r_{Sd\,X}:Sd^{j+1}\,X \to Sd\,X. \notag
\end{align}
Let $r_X:Sd^j\,X \to X$ be defined as a composition of simplicial approximations to the identity where at each stage the barycentres $\widehat{\tau}$ of simplices are sent to whichever vertex of $\tau$ is the closest to a vertex of $X$. Define $r_{Sd\,X}$ the same way but always mapping barycentres towards vertices of $Sd\,X$. Let $P_X:\id_X\simeq r_X$ and $P_{Sd\,X}:\id_{Sd\,X} \simeq r_{Sd\,X}$ be the usual canonically defined homotopies. We play the geometric properties of $r_X$ and $r_{Sd\,X}$ off against each other to obtain the desired results.

Let $C:= \Delta^{n-*}(X)\in \BB(\A_*(X))$ and $D:= \Delta^{lf}_*(Sd\,X) \in \BB(\A^*(Sd\,X))$. Following the proof of Theorem \ref{triangulise} $r_{Sd\, X}$ and $P_{Sd\, X}$ induces a chain equivalence
\begin{displaymath}
\xymatrix@C=2cm{(Sd_{r_{Sd\,X}}\, D, d_{Sd_{r_{Sd\,X}}\, D}, (P_{Sd\, X})_D)\ar@<0.5ex>[r]^-{(r_{Sd\, X)_D}} & (D, d_D, 0)\ar@<0.5ex>[l]^-{(s_{Sd\, X)_D}}}
\end{displaymath}
which assembles to 
\begin{displaymath}
\xymatrix@C=2cm{\ttt_j((Sd_{r_{Sd\,X}}\, D), (d_{Sd_{r_{Sd\,X}}\, D)_{\ttt_j,\ttt_j}}, ((P_{Sd\, X})_D)_{\ttt_j,\ttt_j})\ar@<0.5ex>[r]^-{((r_{Sd\, X})_D)_{\ttt_0,\ttt_j}} & (\ttt_0 D, (d_D)_{\ttt_0,\ttt_0}, 0)\ar@<0.5ex>[l]^-{((s_{Sd\, X})_D)_{\ttt_j,\ttt_0}}}
\end{displaymath}
where $\ttt_j: \BB(\A^*(Sd^{j+1}\, X) \to \BB(\A_*(Sd^j\, X)$ denotes the functor that assembles dual cells in $Sd^{j+1}\, X$. Composing this with the chain equivalence  
\begin{displaymath}
\xymatrix{(C, d_C, 0)\ar@<0.5ex>[r]^-{(r_X)^C} & (Sd_{r_{X}}\, C, d_{Sd_{r_{X}}\,C}, (P_X)^C),\ar@<0.5ex>[l]^-{(s_X)^C} 
}
\end{displaymath}
induced by $(r_X,P_X)$ and the $\ep$-controlled chain equivalence 
\begin{displaymath}
\xymatrix{(Sd_{r_{X}}\, C, d_{Sd_{r_{X}}\, C}, Q_C)\ar@<0.5ex>[r]^-{\phi} & (\ttt_j(Sd_{r_{Sd\,X}}\, D), (d_{Sd_{r_{Sd\,X}}\, D})_{\ttt_j,\ttt_j}, Q_D),\ar@<0.5ex>[l]^-{\psi}
}
\end{displaymath}
which exists by hypothesis this yields the following chain equivalence
\begin{displaymath}
\xymatrix{ (C, d_C, (s_X)^{C}(Q_C + \psi ((P_{Sd\,X})_D)_{\ttt_j, \ttt_j}\phi)(r_X)^C) \ar@<0.5ex>[d]^-{((r_{Sd\, X})_D)_{\ttt_0,\ttt_j}\phi (r_X)^C} \\ (\ttt_0 D,(d_{D})_{\ttt_0,\ttt_0},((r_{Sd\,X})_D)_{\ttt_0,\ttt_j}(Q_D + \phi (P_X)^C \psi)((s_{Sd\,X})_D)_{\ttt_j,\ttt_0}).\ar@<0.5ex>[u]^-{(s_X)^C\phi ((s_{Sd\,X})_D)_{\ttt_j,\ttt_0}}
}
\end{displaymath}
Examining the properties of $r_X$, $r_{Sd\,X}$, $\ttt_0$ and $\ttt_j$ we observe that this is seen to be a chain equivalence in $\BB(\A_*(X))$. The properties we need are that for all $\sigma\in X$:
\begin{align}
 &r_X^{-1}(\mathring{\sigma}) \subset N_{c_X}(Sd^{j-1}\, D(\mathring{\sigma},X)), \label{won} \\
 &\ttt_j(N_\ep(r_X^{-1}(\mathring{\sigma}))) \subset \bigcup_{\tau\in D(\sigma,X)}{r_{Sd\, X}^{-1}(\mathring{\tau})}, \label{twoo} \\
 &(P_X)^C: C[N_{c_{Sd\, X + \ep}}(Sd^{j-1}\, D(\sigma,X))] \to C[N_{c_{Sd\, X + \ep}}(Sd^{j-1}\, D(\sigma,X))], \label{free} \\
 &\bigcup_{\tau\in D(\sigma,X)}{r_{Sd\, X}^{-1}(\mathring{\tau})} \subset \ttt_j(Sd^{j}\,st (\mathring{\sigma}) \backslash N_{3\ep}(\partial (St\, \mathring{\sigma} \backslash st\, \mathring{\sigma}))). \label{vor}
\end{align}

\begin{figure}[h!]
\begin{center}
{
\includegraphics[width=9cm]{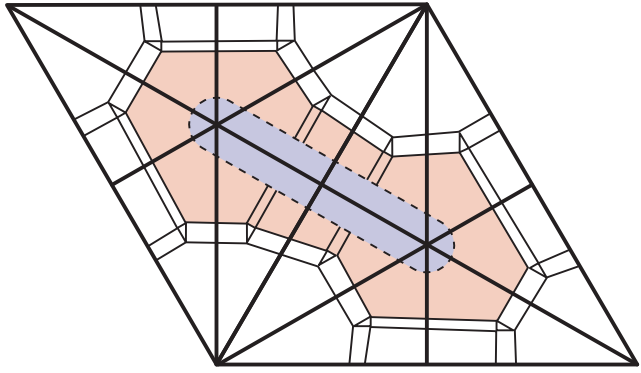}
}
\caption{Illustrating equation $(\ref{twoo})$ with LHS in blue and RHS in red.}
\label{Fig:rXrSdXinterplay}
\end{center}
\end{figure}

The first claim is immediate from the definition of $r_X$; 
all simplices spanned by vertices in $\mathring{D}(v,X)$ are mapped to $v$ hence are in $r_X^{-1}(v)$. All other simplices in $D(v,X)$ are contained within $N_{c_X}(\partial D(v,X))$ so the result follows.

Claim $(\ref{twoo})$ is given by \[\ttt_j(N_\ep(r_X^{-1}(\mathring{\sigma}))) \subset N_{\ep+c_{Sd\,X}}(Sd\, r_X^{-1}(\mathring{\sigma})) \subset N_{\ep + c_{X} + c_{Sd\, X}}(Sd^j\, D(\mathring{\sigma},X)) \subset \bigcup_{\tau\in D(\sigma,X)}{r^{-1}_{Sd\,X}(\mathring{\tau})},\] where the second inclusion follows from $(\ref{won})$ and the third from the fact that \[\ep+ c_{X} + c_{Sd\, X} < \comesh(Sd\, X).\]

Claim $(\ref{free})$ similarly follows from the definition of $r_X$; $(P_X)^C$ maps a region to all simplices whose tracks under $P_X$ go through that region. The region $N_{c_{Sd\, X + \ep}}(Sd^{j-1}\, D(\sigma,X))$ only contains simplices that are furthest from their destination as illustrated by \Figref{Fig:rakefest}.

\begin{figure}[h!]
\begin{center}
{
\includegraphics[width=9cm]{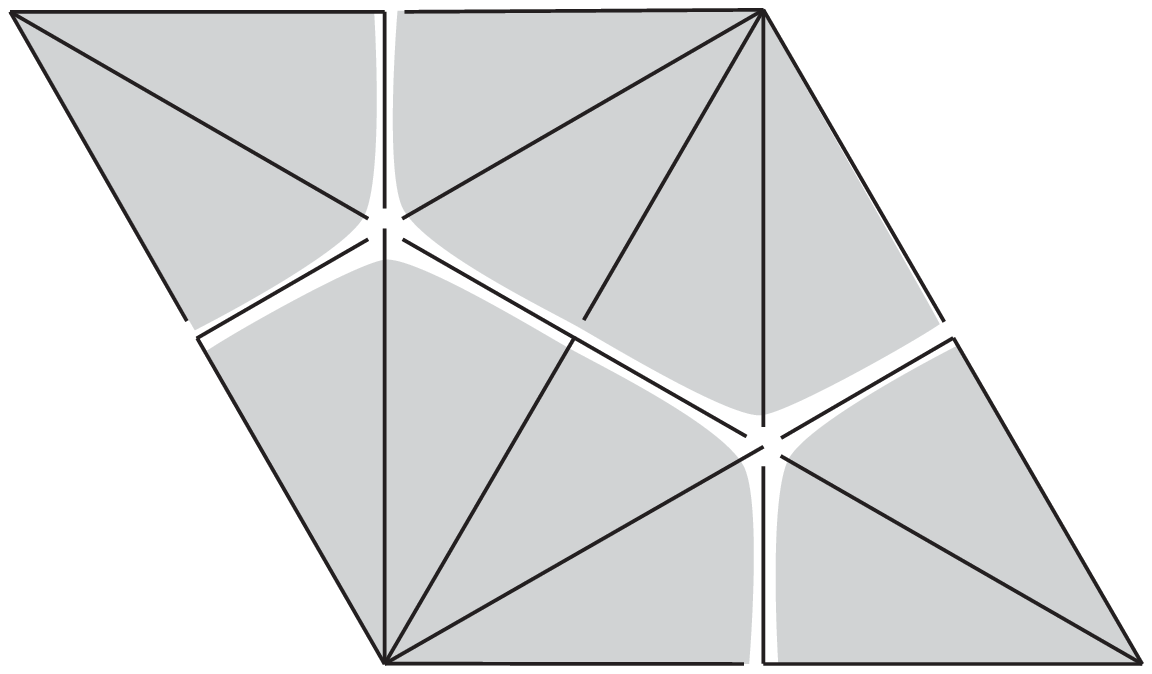} 
}
\caption{}
\label{Fig:rakefest}
\end{center}
\end{figure}

Hence the claim holds.

Claim $(\ref{vor})$ is given by \[\bigcup_{\tau\in D(\sigma,X)}{r_{Sd\, X}^{-1}(\mathring{\tau})} = Sd^j\,st(\mathring{\sigma}) \backslash \bigcup_{\rho \in \partial(St(\mathring{\sigma}))}{r_{Sd\, X}^{-1}(\mathring{\rho})}\subset Sd^{j}\,st (\mathring{\sigma}) \backslash N_{4\ep}(\partial (St\, \mathring{\sigma} \backslash st\, \mathring{\sigma}))\] since $4\ep< \comesh(Sd\, X)$. Thus the claim holds as $\ttt_j$ reduces $4\ep$ to $4\ep-c_{Sd\, X}>3\ep$.

Now we show that the chain equivalence $C\cong \ttt_0 D$ is a chain equivalence in $\BB(\A_*(X))$ by proving that the maps 
\begin{enumerate}[(i)]
 \item $((r_{Sd\, X})_D)_{\ttt_0,\ttt_j}\phi (r_X)^C$, 
 \item $(s_X)^C\psi ((s_{Sd\,X})_D)_{\ttt_j,\ttt_0}$, 
 \item $(s_X)^{C}(Q_C + \psi ((P_{Sd\,X})_D)_{\ttt_j, \ttt_j}\phi)(r_X)^C)$ and
 \item  $((r_{Sd\,X})_D)_{\ttt_0,\ttt_j}(Q_D + \phi (P_X)^C \psi)((s_{Sd\,X})_D)_{\ttt_j,\ttt_0}$
\end{enumerate}
are morphisms of $\A_*(X)$. We may ignore $Q_C$ and $Q_D$ in our calculations as these have control $\ep$ and must be strictly better behaved than  $\psi ((P_{Sd\,X})_D)_{\ttt_j, \ttt_j}\phi$ and $\phi (P_X)^C \psi$ respectively. Using the above observations and arguing by supports as in the proof of Theorem \ref{triangulise} we have
\begin{enumerate}[(i)]
 \item $\xymatrix@R=3mm{ \mathring{\sigma} \ar[r]^-{(r_X)^C} & r_X^{-1}(\mathring{\sigma}) \ar[r]^-{\phi} & \ttt_j(N_\ep(r_X^{-1}(\mathring{\sigma}))) \subset \bigcup_{\tau\in D(\sigma,X)}{r_{Sd\, X}^{-1}(\mathring{\tau})} \\ &  \ar[r]^-{(r_{Sd\, X})_D} & D(\sigma,X) = \ttt_0(st\, \mathring{\sigma}), 
}$
\item $\xymatrix@R=3mm{ \ttt_0(\mathring{\sigma}) = \mathring{D}(\sigma,X) \ar[rr]^-{(s_{Sd\,X})_D} && Sd^j\, D(\sigma,X) \subset \ttt_j(N_{c_{Sd\, X}}(Sd^{j-1}\, D(\sigma,X))) \\
&\ar[r]^-{\psi} & N_{\ep + c_{Sd\, X}}(Sd^{j-1}\, D(\sigma,X)) \subset Sd^j\, (st\, \mathring{\sigma}) \ar[r]^-{(s_X)^C} & st\,\mathring{\sigma},
}$
\item $\xymatrix@R=3mm{ \mathring{\sigma} \ar[r]^-{(r_X)^C} & r_X^{-1}(\mathring{\sigma}) \ar[r]^-{\phi} & \ttt_j(N_\ep(r_X^{-1}(\mathring{\sigma}))) \subset \bigcup_{\tau\in D(\sigma,X)}{r_{Sd\, X}^{-1}(\mathring{\tau})} \\ & \ar[r]^-{(P_{Sd\,X})_D} & \bigcup_{\tau\in D(\sigma,X)}{r_{Sd\, X}^{-1}(\mathring{\tau})} \subset \ttt_j(Sd^j\,st (\mathring{\sigma}) \backslash N_{3\ep}(\partial (St\, \mathring{\sigma} \backslash st\, \mathring{\sigma}))) \\ & \ar[r]^-{\psi} & Sd^j\,st (\mathring{\sigma}) \backslash N_{2\ep}(\partial (St\, \mathring{\sigma} \backslash st\, \mathring{\sigma})) \subset Sd^j\,st(\mathring{\sigma}) \ar[r]^-{(s_X)^{C}} & st(\mathring{\sigma}),
}$
\item $\xymatrix@R=3mm{ \ttt_0(\mathring{\sigma}) = \mathring{D}(\sigma,X) \ar[rr]^-{(s_{Sd\,X})_D} && Sd^j\, D(\sigma,X) \subset \ttt_j(N_{c_{Sd\, X}}(Sd^{j-1}\, D(\sigma,X))) }$

$\xymatrix@R=3mm{\ar[r]^-{\psi} & N_{\ep + c_{Sd\, X}}(Sd^{j-1}\, D(\sigma,X)) \ar[r]^-{(P_X)^C} & N_{\ep + c_{Sd\, X}}(Sd^{j-1}\, D(\sigma,X)) }$

$\xymatrix@R=3mm{ \ar[r]^-{\phi} & N_{2\ep + c_{Sd\, X}}(Sd^{j-1}\, D(\sigma,X)) \subset \bigcup_{\tau\in D(\sigma,X)}{r_{Sd\, X}^{-1}(\mathring{\tau})} \ar[rr]^-{(r_{Sd\, X})_D} && D(\sigma,X) = \ttt_0(st\,\mathring{\sigma}).}$
\end{enumerate}
This completes the proof.
\end{proof}

We also have a Poincar\'{e} duality splitting theorem.
\begin{rmk}\label{exptransequal}
Note that
\begin{displaymath}
\xymatrix@R=0mm@C=3mm{
(t)_* \ttt \Delta^{lf}_*(t^1(X\times\R)) \ar@{}[r]|-{=} & \ttt \Delta^{lf}_*(Sd_1\, t^1(X\times\R)).
} 
\end{displaymath}
\end{rmk}

\begin{thm}
If $X\times\R$ has bounded $(n+1)$-dimensional Poincar\'{e} duality measured in $O(X_+)$ then $X$ is a $\BB(\A_*(Sd^i\,X))$-controlled Poincar\'{e} complex, for all $i\in \N$.
\end{thm}

\begin{proof}
By assumption there is a $B<\infty$ and Poincar\'{e} duality chain equivalences \[\phi = [t^1(X\times\R)]\cap -: \Delta^{n+1-*}(t^1(X\times\R)) \to \ttt\Delta^{lf}_*(Sd\, t^1(X\times\R))\] in $\GG_{t^1(X\times\R)}(\A)$ with control $B$.

If $B=0$, $\phi$ is diagonal and hence a chain equivalence in $\BB(\A_*(t^1(X\times\R)))$. The restriction to $t^1(X\times\{v_i\})$ is hence a chain equivalence 
\begin{equation}\label{neededtosplit}
\Delta^{n+1-*}(t^1(X\times\{v_i\})) \to \ttt\Delta^{lf}_*(Sd\, t^1(X\times\R))|_{t^1(X\times\{v_i\})} 
\end{equation}
in $\BB(\A_*(t^1(X\times\R))) = \BB(\A_*(Sd^{i}\,X))$. 

Since for all $\tau\in Sd^i\,X$: \[ \mathring{D}(\tau\times\{v_i\},t^1(X\times\R))\cong \mathring{D}(\tau,Sd^{i}\,X)\times \mathring{D}(\{v_i\},\R)\] we deduce that \[ \ttt\Delta^{lf}_*(Sd\, t^1(X\times\R))|_{t^1(X\times\{v_i\})} \simeq \ttt\Delta^{lf}_{*-1}(Sd^{i+1}\,X)\] in $\BB(\A_*(Sd^{i}\,X))$ and hence \[\Delta^{n-(*-1)}(Sd^i\,X) \to \ttt\Delta_{*-1}^{lf}(Sd^{i+1}\,X)\] is a chain equivalence in $\BB(\A_*(Sd^i\,X))$ as claimed.

So suppose $B>0$. For all $v_j>B+1$ we can find an interval $J:=[v_{j_-},v_{j_+}]\subset [1,\infty)$ containing $(v_j-B, v_j+B)$. Since $t^1(X\times J)$ is f.d. and l.f. it satisfies the conditions of the Poincar\'{e} duality squeezing theorem; let \[\ep = \ep(t^1(X\times J)),\quad i = i(t^1(X\times J))\] be as given by this theorem.

By Remark \ref{slide} there exists a $k\geqslant i$ such that \[t^{-k}\phi: t^{-k}\Delta^{n+1-*}(t^1(X\times\R))\to t^{-k}\ttt\Delta^{lf}_*(Sd\, t^1(X\times\R))\] is a chain equivalence in $\mathbb{G}_{t^{-k}(X\times\R)}(\A)$ with bound less than $\frac{\ep}{3}$. By Example \ref{exptransequal}
we have that 
\begin{align}
 t^{-k}\Delta^{n+1-*}(t^1(X\times\R)) &= Sd_k \Delta^{n+1-*}(t^1(X\times\R)) \notag \\
 t^{-k}\ttt\Delta^{lf}_*(Sd\, t^1(X\times\R)) &= \ttt Sd_k \Delta^{lf}_*(Sd\, t^1(X\times\R)), \notag
\end{align}
so we view $t^{-k}\phi$ as an $\frac{\ep}{3}$-controlled chain equivalence \[Sd_k \Delta^{n+1-*}(t^1(X\times\R)) \to 
\ttt Sd_k \Delta^{lf}_*(Sd\, t^1(X\times\R)).\]
Proceeding as in Theorem \ref{PDsqueezing} we form the composition \[((r_{Sd\, t^1(X\times\R)})_D)_{\ttt_0,\ttt_i} t^{-k}\phi (r_{t^1(X\times\R)})^C: \Delta^{n+1-*}(t^1(X\times\R)) \to \ttt \Delta^{lf}_*(Sd\, t^1(X\times\R))\] for  
$C=\Delta^{n+1-*}(t^1(X\times\R))$ and $D= \Delta^{lf}_*(Sd\, t^1(X\times\R))$. The composition is a chain equivalence which, by the proof of Theorem \ref{PDsqueezing}, restricts to a $\BB(\A_*(t^1(X\times J)))$ chain equivalence in a neighbourhood of $t^1(X\times\{v_j\})$. Hence we get equation (\ref{neededtosplit}) as in the $B=0$ case and the result follows.
\end{proof}

\section*{Appendix}\label{section:appendix}

\begin{lem}\label{euclideancomesh}
Let $\sigma \subset \R^N$ be an simplex linearly embedded in euclidean space. Then
\[ \comesh(Sd\, \sigma) \geqslant \dfrac{\rad(\sigma)}{|\sigma|(|\sigma|+1)}.\]
\end{lem}
\begin{proof}
All the edges of $\sigma$ have length at least $2\rad(\sigma)$ hence $\sigma$ contains a 
regular $|\sigma|$-simplex $\tau$ with edges all of length $2\rad(\sigma)$ inside it. As $\tau \subset \sigma$ we must have $\comesh(Sd\, \tau)\leqslant \comesh(Sd\, \sigma)$. As $\tau$ is regular, the length of the shortest edge in $Sd\, \tau$ is equal to $\rad(\tau)$ which is \[\dfrac{\rad(\sigma)\sqrt{2}}{\sqrt{|\sigma|(|\sigma|+1)}}.\] Thus each simplex $\rho \in Sd\,\tau$ contains a regular $|\rho|$-simplex with edge length equal to \[\dfrac{\rad(\sigma)\sqrt{2}}{\sqrt{|\sigma|(|\sigma|+1)}}.\] This regular $|\rho|$-simplex thus has radius \[\dfrac{\rad(\sigma)\sqrt{2}}{\sqrt{|\sigma|(|\sigma|+1)}}\dfrac{1}{\sqrt{2|\rho|(|\rho|+1)}} \geqslant \dfrac{\rad(\sigma)}{|\sigma|(|\sigma|+1)}\]so the result follows.
\end{proof}

\begin{lem}
The maps $s_*$ and $r_*$ defined in the proof of Theorem \ref{Thm:subdivassemblechainequiv} are chain maps. 
\end{lem}

\begin{proof}
In the following, for a statement $S$, $\mathbf{1}_{\{S\}}$ will denote the indicator function:\[ \mathbf{1}_{\{S\}} := \brcc{1,}{S\; \mathrm{true}}{0,}{S\;\mathrm{false}.}\] 
In verifying $s_*$ is a chain map we split into four cases:

\underline{$\widehat{\sigma}*r(\widehat{\sigma})\leqslant \widetilde{\rho}$}: Let $\widetilde{\rho} = \widehat{\sigma}*\tildetilde{\sigma}*\tildetilde{\rho}$ where $r(\widehat{\sigma})\in\tildetilde{\sigma}<\sigma$ and $\tildetilde{\rho}\in \link(\sigma,X)$. We consider the component \[(s_*d_C - d_{Sd_r\,C}s_*)_{\widetilde{\rho},\tau,n}\] and show it is zero for all $\tau$. Note that it is trivially zero unless $\tau = \sigma*\tildetilde{\tau}$ for some $\tildetilde{\rho}\leqslant \tildetilde{\tau}$ so we assume this in the following.

\begin{align}
(s_*d_C - d_{Sd_r\,C}s_*)_{\widetilde{\rho},\tau,n} 
&= \sum_{\tildetilde{\rho}\leqslant \tildetilde{\rho^\prime} \leqslant \tildetilde{\tau}}{(s_*)_{\widehat{\sigma}*\tildetilde{\sigma}*\tildetilde{\rho},\sigma*\tildetilde{\rho^\prime},n-1}(d_C)_{\sigma*\tildetilde{\rho^\prime},\sigma*\tildetilde{\tau},n}} \notag \\ 
&-\sum_{\begin{array}{c}{\tildetilde{\sigma}\leqslant \tildetilde{\sigma^\prime} < \sigma} \\ {\tildetilde{\rho}\leqslant \tildetilde{\rho^\prime} \leqslant \tildetilde{\tau}} \end{array}} (d_{Sd_r\, C})_{\widehat{\sigma}*\tildetilde{\sigma}* \tildetilde{\rho}, \widehat{\sigma}*\tildetilde{\sigma^\prime}*\tildetilde{\rho^\prime},n}(s_*)_{\widehat{\sigma}*\tildetilde{\sigma^\prime}*\tildetilde{\rho^\prime},\sigma* \tildetilde{\tau},n} \notag \\
&= \sum_{\begin{array}{c}{\tildetilde{\sigma}\leqslant \tildetilde{\sigma^\prime} = \sigma} \\ {\tildetilde{\rho}\leqslant \tildetilde{\rho^\prime} \leqslant \tildetilde{\tau}} \end{array}}(-1)^n (d_C)_{\tildetilde{\sigma}*\tildetilde{\rho},\sigma * \tildetilde{\rho^\prime}, \sigma*\tildetilde{\tau},n} -\sum_{\begin{array}{c}{\tildetilde{\sigma}\leqslant \tildetilde{\sigma^\prime} < \sigma} \\ {\tildetilde{\rho}\leqslant \tildetilde{\rho^\prime} \leqslant \tildetilde{\tau}} \end{array}}(-1)^{n+1}(d_C)_{\tildetilde{\sigma}*\tildetilde{\rho},\tildetilde{\sigma^\prime} * \tildetilde{\rho^\prime}, \sigma*\tildetilde{\tau},n} \notag \\ 
&= (-1)^n \sum_{\tildetilde{\sigma}*\tildetilde{\rho} \leqslant \rho \leqslant \sigma * \tildetilde{\tau}}(d_C)_{\tildetilde{\sigma}*\tildetilde{\rho},\rho,\sigma*\tildetilde{\tau},n} = (-1)^n(d_C^2)_{\tildetilde{\sigma}*\tildetilde{\rho},\sigma*\tildetilde{\tau}}=0. \notag
\end{align}

\underline{$\widehat{\sigma}\in \widetilde{\rho}$, $r(\widehat{\sigma})\notin \widetilde{\rho}$}: Let $\widetilde{\rho} = \widehat{\sigma}*\tildetilde{\sigma} * \tildetilde{\rho}$ with $r(\widehat{\sigma})\notin \tildetilde{\sigma}<\sigma$, $\tildetilde{\rho}\in \link(\sigma,X)$. Again for non-triviality we assume that $\tau = \sigma * \tildetilde{\tau}$ for $\tildetilde{\rho}\leqslant \tildetilde{\tau}$. Let $r^{-1}(\sigma)$ denote the simplex $\widehat{\sigma}*r(\widehat{\sigma})^\perp$ where $r(\widehat{\sigma})^\perp$ is the codimension subsimplex of $\sigma$ not containing $r(\widehat{\sigma})$.

\begin{align}
(s_*d_C - d_{Sd_r\,C}s_*)_{\widetilde{\rho},\tau,n} 
&= \mathbf{1}_{\{r(\widehat{\sigma}*\tildetilde{\sigma}) = \sigma\}}(d_C)_{\sigma*\tildetilde{\rho},\sigma*\tildetilde{\tau},n} \notag \\
&- (1 - \mathbf{1}_{\{r(\widehat{\sigma}*\tildetilde{\sigma})= \sigma\}})(d_{Sd_r\,C})_{\widehat{\sigma}*\tildetilde{\sigma}*\tildetilde{\rho},\widehat{\sigma}*r(\widehat{\sigma}) * \tildetilde{\sigma}*\tildetilde{\rho},n}(s_*)_{\widehat{\sigma}*r(\widehat{\sigma}) * \tildetilde{\sigma}*\tildetilde{\rho}, \sigma*\tildetilde{\tau}, n} \notag \\
&- \sum_{\sumlines{\tildetilde{\sigma}\leqslant \tildetilde{\sigma^\prime}\leqslant r(\widehat{\sigma})^\perp}{\tildetilde{\rho}\leqslant \tildetilde{\rho^\prime}\leqslant \tildetilde{\tau}}} (d_{Sd_r\,C})_{\widehat{\sigma}*\tildetilde{\sigma}*\tildetilde{\rho},\widehat{\sigma}*\tildetilde{\sigma^\prime}*\tildetilde{\rho^\prime},n}(s_*)_{\widehat{\sigma}*\tildetilde{\sigma^\prime}*\tildetilde{\rho^\prime},\sigma*\tildetilde{\tau},n} \notag \\
&= \mathbf{1}_{\{r(\widehat{\sigma}*\tildetilde{\sigma})= \sigma\}} (d_C)_{\sigma*\tildetilde{\rho},\sigma*\tildetilde{\tau},n} - (-1)^n(-1)^{n+1}(1 - \mathbf{1}_{\{r(\widehat{\sigma}*\tildetilde{\sigma})= \sigma\}})(d_C)_{r(\widehat{\sigma})*\tildetilde{\sigma}*\tildetilde{\rho},\sigma*\tildetilde{\tau},n} \notag \\
&- \sum_{\tildetilde{\rho}\leqslant \tildetilde{\rho^\prime}\leqslant \tildetilde{\tau}}(d_{Sd_r\,C})_{\widehat{\sigma}*\tildetilde{\sigma}*\tildetilde{\rho}, r^{-1}(\sigma)*\tildetilde{\rho^\prime},n}(s_*)_{r^{-1}(\sigma)*\tildetilde{\rho^\prime}, \sigma*\tildetilde{\tau},n} = 0. \notag
\end{align}
Where we note that $(d_C)_{\sigma*\tildetilde{\rho},\sigma*\tildetilde{\tau},n} = (d_C)_{r(\widehat{\sigma})*\tildetilde{\sigma}*\tildetilde{\rho},\sigma*\tildetilde{\tau},n}$ in the case that $r(\widehat{\sigma}*\tildetilde{\sigma}) = \sigma$ and the last term only contributes $-(d_C)_{r(\widehat{\sigma})*\tildetilde{\sigma}*\tildetilde{\rho}, \sigma*\tildetilde{\tau},n}$ from the $\tildetilde{\rho^\prime} = \tildetilde{\tau}$ term in the sum.

\underline{$\widehat{\sigma}\notin \widetilde{\rho}$, $r(\widehat{\sigma})\in \widetilde{\rho}$}: As before we write $\widetilde{\rho} = \tildetilde{\sigma} * \tildetilde{\rho}$ where $r(\widetilde{\sigma})\in \tildetilde{\sigma}<\sigma$ and $\tildetilde{\rho}\in \link(\sigma,X)$.

The only non-trivial $\tau$ to check are $\tau = \sigma * \tildetilde{\tau}$ for $\tildetilde{\rho}\leqslant \tildetilde{\tau}$ and $\tau = \tildetilde{\tau}\geqslant \widetilde{\rho}$. So suppose the former:

\begin{align}
(s_*d_C - d_{Sd_r\,C}s_*)_{\widetilde{\rho},\tau,n} 
&= (s_*)_{\tildetilde{\sigma}*\tildetilde{\rho}, \tildetilde{\sigma}*\tildetilde{\rho}, n-1}(d_C)_{\tildetilde{\sigma}*\tildetilde{\rho}, \sigma*\tildetilde{\tau},n} \notag \\
&- (d_{Sd_r\,C})_{\tildetilde{\sigma}*\tildetilde{\rho}, \widehat{\sigma}*\tildetilde{\sigma}*\tildetilde{\rho}, n}(s_*)_{\widehat{\sigma}*\tildetilde{\sigma}*\tildetilde{\rho}, \sigma*\tildetilde{\tau}, n} \notag \\
&- \sum_{\sumlines{\tildetilde{\sigma}\leqslant \tildetilde{\sigma^\prime}<\sigma}{\tildetilde{\rho}\leqslant \tildetilde{\rho^\prime}\leqslant \tildetilde{\tau}}}(d_{Sd_r\,C})_{\tildetilde{\sigma}*\tildetilde{\rho},\tildetilde{\sigma^\prime}*\tildetilde{\rho^\prime},n}(s_*)_{\tildetilde{\sigma^\prime}*\tildetilde{\rho^\prime},\sigma*\tildetilde{\tau},n} \notag \\
&= (d_C)_{\tildetilde{\sigma}*\tildetilde{\rho}, \sigma*\tildetilde{\tau},n} -0 \notag \\
&- (d_{Sd_r\,C})_{\tildetilde{\sigma}*\tildetilde{\rho},r^{-1}(\sigma)*\tildetilde{\tau},n}(s_*)_{r^{-1}(\sigma)*\tildetilde{\tau},\sigma*\tildetilde{\tau},n} = 0. \notag
\end{align}
Here the second term contributes $0$ since $(s_*)_{\tildetilde{\sigma}*\tildetilde{\rho}, \tildetilde{\sigma}*\tildetilde{\rho}, n-1} \neq 0$ only if $\tildetilde{\sigma} = r(\widehat{\sigma})^\perp$ which is not possible as $r(\widehat{\sigma})\in \tildetilde{\sigma}$.

Suppose now the latter: $\tau = \tildetilde{\tau}\geqslant \widetilde{\rho}$. Then
\begin{align}
(s_*d_C - d_{Sd_r\,C}s_*)_{\widetilde{\rho},\tau,n} 
&= (s_*)_{\widetilde{\rho},\widetilde{\rho},n-1}(d_C)_{\widetilde{\rho},\tau,n} - \sum_{\widetilde{\rho}\leqslant \rho \leqslant \tau}(d_{Sd_r\,C})_{\widetilde{\rho},\rho,n}(s_*)_{\rho,\tau,n} \notag \\
&= (d_C)_{\widetilde{\rho},\tau,n} - (d_C)_{\widetilde{\rho},\tau,n} = 0. \notag
\end{align}

\underline{$\widehat{\sigma}\notin \widetilde{\rho}$, $r(\widehat{\sigma})\notin \widetilde{\rho}$}: For a non-trivial computation we assume that 
$\widetilde{\rho}\leqslant \tau \in X$, but $\tau$ may or may not also be a simplex in $X^\prime$ dependent on whether it contains $\sigma$ or not.

\begin{align}
(s_*d_C - d_{Sd_r\,C}s_*)_{\widetilde{\rho},\tau,n} 
&= (s_*)_{\widetilde{\rho},\widetilde{\rho},n-1}(d_C)_{\widetilde{\rho},\tau,n} \notag \\
&-\mathbf{1}_{\{\tau \in X^\prime\}} (d_{Sd_r\,C})_{\widetilde{\rho},\tau,n}(s_*)_{\tau,\tau,n} \notag \\
&-\mathbf{1}_{\{\tau = \sigma * \tildetilde{\tau} \}} (d_{Sd_r\,C})_{\widetilde{\rho},\widehat{\sigma}*r(\widehat{\sigma})^\perp *\tildetilde{\tau},n}(s_*)_{\widehat{\sigma}*r(\widehat{\sigma})^\perp *\tildetilde{\tau},\sigma * \tildetilde{\tau},n} \notag \\
&= (d_C)_{\widetilde{\rho},\tau,n} - \mathbf{1}_{\{\tau \in X^\prime\}}(d_C)_{\widetilde{\rho},\tau,n} - \mathbf{1}_{\{\tau \notin X^\prime\}}(d_C)_{\widetilde{\rho},\tau,n} = 0. \notag
\end{align}

Next we verify that $r_*$ chain map by showing that $(r_*d_{Sd_r\,C} - d_C r_*)_{\tau,\widetilde{\rho},n} = 0$ for all $\tau\in X$, $\widetilde{\rho}\in X^\prime$:
\begin{align}
 (r_*d_{Sd_r\,C} - d_C r_*)_{\tau,\widetilde{\rho},n} =&\; \mathbf{1}_{\{\widehat{\sigma}r(\widehat{\sigma})\in \widetilde{\rho}\}}\mathbf{1}_{\{r(\widetilde{\rho}) = \tau = r(\widehat{\sigma})*\tildetilde{\tau}\}}(r_*)_{\tau,\tau,n-1}(d_{Sd_r\,C})_{\tau,\widetilde{\rho},n} \notag \\
 &+ \mathbf{1}_{\{\widehat{\sigma}r(\widehat{\sigma})\in \widetilde{\rho}\}}\mathbf{1}_{\{r(\widetilde{\rho}) = \tau = r(\widehat{\sigma})*\tildetilde{\tau}\}}(r_*)_{\tau,\widehat{\sigma}*\tildetilde{\tau}}(d_{Sd_r\,C})_{\widehat{\sigma}*\tildetilde{\tau},\widetilde{\rho},n} \notag \\
 &+ \mathbf{1}_{\{\tau \leqslant r(\widetilde{\rho})\}} \bigcup_{\widetilde{\rho}^\prime \leqslant \widetilde{\rho}}\mathbf{1}_{\{\widehat{\sigma}r(\widehat{\sigma})\notin \widetilde{\rho}\}} (r_*)_{\tau,\widetilde{\rho}^\prime,n-1}(d_{Sd_r\, C})_{\widetilde{\rho}^\prime, \widetilde{\rho},n} \notag \\
 &+ \mathbf{1}_{\{\tau \leqslant r(\widetilde{\rho})\}}\mathbf{1}_{\{\widehat{\sigma}r(\widehat{\sigma})\notin \widetilde{\rho}\}}(d_C)_{\tau,r(\widetilde{\rho}),n}(r_*)_{r(\widetilde{\rho}),\widetilde{\rho},n} \notag \\
 =& \; \mathbf{1}_{\{\widehat{\sigma}r(\widehat{\sigma})\in \widetilde{\rho}\}}\mathbf{1}_{\{r(\widetilde{\rho}) = \tau\}} \left( (-1)^{n+1}\id_{C(\tau)} + (-1)^n \id_{C(\tau)}\right) \notag \\
 &+ \mathbf{1}_{\{\tau \leqslant r(\widetilde{\rho})\}}\mathbf{1}_{\{\widehat{\sigma}r(\widehat{\sigma})\notin \widetilde{\rho}\}} \left((d_C)_{\tau,r(\widetilde{\rho}),n} - (d_C)_{\tau,r(\widetilde{\rho}),n} \right) \notag \\
 =& \; 0 \notag
\end{align}
\end{proof}

\bibliographystyle{hep}
\bibliography{spirosbib}{}
\end{document}